\documentclass[11pt,a4paper]{article}
\usepackage{hyperref}
\usepackage[english]{babel}
\usepackage[utf8]{inputenc}
\usepackage{amsmath,amsthm,amssymb,amsfonts}
\usepackage{enumerate}
\usepackage{faktor}
\usepackage{dsfont}
\usepackage{textcomp}
\usepackage{graphicx}
\usepackage{extarrows}
\usepackage{epigraph}
\usepackage{graphicx}
\usepackage{subfigure}
\usepackage{sidecap}
\usepackage{indentfirst}
\usepackage{tikz}
\usepackage{pgfplots}
\usetikzlibrary{arrows,matrix,patterns,decorations.markings,positioning}
\usepackage{caption}
\usepackage{float}
\usepackage{color}
\usepackage{epstopdf}
\usepackage{epsfig}
\usepackage{stmaryrd}
\usepackage{color}

\newcommand{\R}{\mathbb{R}}

\newcommand{\Z}{\mathbb{Z}}
\newcommand{\Q}{\mathbb{Q}}
\newcommand{\F}{\mathbb{F}}
\newcommand{\C}{\mathbb C}
\newcommand{\T}{\mathbb T}

\newcommand{\dd}{\mathrm{d}}

\newcommand{\xist}{\xi_{\text{st}}}

\renewcommand{\geq}{\geqslant}

\newtheorem{teo}{Theorem}[section]
\newtheorem*{teo*}{Theorem}
\newtheorem{lemma}[teo]{Lemma}
\newtheorem{prop}[teo]{Proposition}
\newtheorem*{prop*}{Proposition}

\newtheorem{cor}[teo]{Corollary}
\newtheorem{conj}[teo]{Conjecture}

\theoremstyle{definition}
\newtheorem{remark}[teo]{Remark}

\pgfplotsset{compat=1.13}
\begin{document}
\title{An invariant of Legendrian and transverse links from open book decompositions of contact 3-manifolds}
\author{\scshape{Alberto Cavallo}\\ \\
 \footnotesize{Alfr\'ed R\'enyi Institute of Mathematics,}\\
 \footnotesize{Budapest 1053, Hungary}\\ \\ \small{cavallo\_alberto@phd.ceu.edu}}
\date{}

\maketitle

\begin{abstract}
 We introduce a generalization of the Lisca-Ozsv\'ath-Stipsicz-Szab\'o
 Legendrian invariant $\mathfrak L$ to links in every rational homology sphere, using
 the collapsed version of link Floer homology.
 
 We represent a Legendrian link $L$ in a contact 3-manifold $(M,\xi)$ with a diagram $D$, given by an open book decomposition 
 of $(M,\xi)$ adapted to $L$, and we construct a chain 
 complex $cCFL^-(D)$ with a special cycle in it denoted by $\mathfrak L(D)$. Then, given two 
 diagrams $D_1$ and $D_2$ which represent
 Legendrian isotopic links, we prove that there is a map between the corresponding
 chain complexes, that induces an isomorphism in homology and sends $\mathfrak L(D_1)$ into $\mathfrak L(D_2)$.
 
 Moreover, a connected sum formula is also
 proved and we use it to give some applications about non-loose Legendrian links; that are
 links such that the restriction of $\xi$ on their complement is tight.  
\end{abstract}
\section{Introduction}
The Legendrian invariant $\mathfrak L$ was first introduced in \cite{LOSS}. In this paper we generalize $\mathfrak L$ to
Legendrian links in rational homology contact 3-spheres. For sake of simplicity, we consider only null-homologous links,
which in this settings are links whose all components represent trivial classes in homology.

A \emph{contact structure} $\xi$ on a differentiable
3-manifold $M$ is a smooth 2-plane field that is given as the kernel of a 1-form $\alpha$ such that
$\alpha\wedge\dd\alpha>0$.
A \emph{Legendrian $n$-component link} $L\hookrightarrow(M,\xi)$ is a link such that the tangent space $T_pL$ 
is contained in $\xi_p$ for every point $p$. A link $T\hookrightarrow(M,\xi)$ is \emph{transverse} if
$T_pT\oplus\xi_p=T_pM$ for every $p\in T$ and $\alpha>0$ on $T$.
A \emph{Legendrian isotopy} between $L$ and $L'$ 
is a smooth isotopy $F_t$ of $M$, where $t\in[0,1]$, such that $F_0(L)=L$, $F_1(L)=L'$ and $F_t(L)$ is Legendrian
for every $t$.

A 3-manifold $M$ can be represented with an open book decomposition $(B,\pi)$.
Suppose that $L$ is a Legendrian $n$-component \emph{oriented} link in $M$, equipped with a contact structure $\xi$.
When some compatibility conditions
with $\xi$ and $L$, which are explained in Section \ref{section:two}, are satisfied $(B,\pi)$, 
together with an appropriate set $A$ of 
embedded arcs in the page $S_1$, is said to be adapted to the Legendrian link $L$. In particular, we prove the
following theorem.
\begin{teo}
  \label{teo:reduced}
  Every Legendrian oriented link $L$ in a contact $3$-manifold $(M,\xi)$ admits an adapted open book decomposition $(B,\pi,A)$
  which is compatible with
  the triple $(L,M,\xi)$. 
  Moreover, the contact framing of $L$ coincides with the framing induced on 
  $L$ by the page $S_1$.
\end{teo}
Suppose from now on that $M$ is a rational homology sphere.
A \emph{multi-pointed Heegaard diagram} is 
defined in \cite{Ozsvath2} as two collections of $g+n-1$ curves
and two sets of $n$ basepoints in a genus $g$ surface $\Sigma$ which describe an oriented link in $M$. 
As shown in Section \ref{section:three}, 
one can associate, up to isotopy, a diagram of this kind to every adapted open book decomposition, compatible with the triple
$(L,M,\xi)$. This diagram is called a Legendrian Heegaard diagram and it is denoted with $D_{(B,\pi,A)}$.
The surface $\Sigma$ in $D_{(B,\pi,A)}$ is obtained by gluing the pages $S_1$ and $\overline{S_{-1}}$ together, moreover
all the basepoints are contained in $S_1$.

In Heegaard Floer theory the diagram $D_{(B,\pi,A)}$ gives a bigraded 
chain complex $\left(cCFL^-(D_{(B,\pi,A)}),\partial^-\right)$,
generated by some discrete subsets of points in $\Sigma$ \cite{Ozsvath} and 
whose homology is an $\F[U]$-module, where $\F$ is the field
with two elements, called \emph{collapsed link Floer homology}. The isomorphism type of such a group is a smooth link invariant 
and it is denoted by $cHFL^-\left(\overline M,L\right)$, see \cite{Book,Ozsvath2}.
Furthermore, there is only one cycle in $cCFL^-\left(D_{(B,\pi,A)}\right)$, see \cite{Matic,LOSS}, 
that lies on the page $S_1$: this cycle is denoted by $\mathfrak L\left(D_{(B,\pi,A)}\right)$. 
\begin{teo}
 \label{teo:main}
 Let us consider a Legendrian Heegaard diagram $D$, given by an open book compatible with a triple $(L,M,\xi)$, where $M$ is a 
 rational homology $3$-sphere and $L$
 is a null-homologous Legendrian $n$-component oriented link. 
 
 Let us
 take the cycle $\mathfrak L(D)\in cCFL^-(D,\mathfrak t_{\mathfrak L(D)})$; if $D_1$ and $D_2$ are diagrams as before for Legendrian isotopic links then we can find a map
 \[cCFL^-(D_1,\mathfrak t_{\mathfrak L(D_1)})\longrightarrow cCFL^-(D_2,\mathfrak t_{\mathfrak L(D_2)})\quad\text{ such that }\quad\mathfrak L(D_1)\longrightarrow\mathfrak L(D_2)\] and inducing a bigraded isomorphism in homology.
 Furthermore, the $\text{Spin}^c$ structure $\mathfrak t_{\mathfrak L(D)}$ coincides with $\mathfrak t_{\xi}$.
\end{teo}
As in \cite{LOSS}, we introduce an equivalence relation on the family of bigraded $\F[U]$-modules with a distinguished element in them. Namely, we say that $(M,m)\sim(N,n)$ if and only if there is a bigraded isomorphism $M\rightarrow N$ such that $m\mapsto n$.
Hence, Theorem \ref{teo:main} tells us that the equivalence class of the pair \[(cHFL^-\left(\overline M,L,\mathfrak t_{\xi}\right),\:[\mathfrak L(D)])\] with respect to $\sim$ is a Legendrian invariant of $(L,M,\xi)$; and we denote it with $\mathfrak L(L,M,\xi)$.

In \cite{Baldwin} Baldwin, Vela-Vick and V\'ertesi, using a different construction, introduced another invariant of Legendrian links in contact 3-manifolds 
which generalizes $\mathfrak L$ in the case of
knots in the standard 3-sphere. The same argument in \cite{Baldwin} implies that this invariant coincides with our
$\mathfrak L$ for every Legendrian link in $(S^3,\xi_{\text{st}})$.
\begin{conj}
 The invariant $\mathfrak L$ concides with the invariant from \emph{\cite{Baldwin}} for every null-homologous link in a rational homology $3$-sphere.
\end{conj}
\begin{remark}
 \label{remark:1}
 It is important to observe that we cannot define $\mathfrak L$ just as a fixed element in $cHFL^-$ without a naturality property, similar to the one given by Juh\'asz, Thurston and Zemke in \cite{JTZ}. In fact, even for Legendrian isotopic links the homology groups have different presentations, according to the choice of the corresponding Legendrian Heegaard diagrams: this fact justifies the statement of Theorem \ref{teo:main} and the definition of $\mathfrak L$. 
 Nonetheless, in the rest of the paper, for the sake of simplicity, we will sometimes reduce the formalism and write $\mathfrak L(L,M,\xi)$ for the homology class $[\mathfrak L(D)]$ in $cHFL^-\left(\overline M,L,\mathfrak t_{\xi}\right)$.
 
 Clearly, it is very difficult to claim that two Legendrian links have different invariant $\mathfrak L$ in general; nonetheless, we can still prove that $\mathfrak L$ has some properties which do not involve naturality: for example its bigrading and $U$-torsion order in $cHFL^-$ and the behaviour under connected sums. Such properties are all discussed in detail in the paper. 
 
 In addition, we note that in \cite{JMZ} Juh\'asz, Miller and Zemke showed that the transverse link invariant introduced in \cite{Baldwin} has indeed a naturality property, but we recall that such an invariant is known to coincide with $\mathfrak L$ only in $(S^3,\xist)$.
\end{remark}
The smooth link type is clearly a Legendrian invariant. 
Together with the Thurston-Bennequin and the rotation numbers are called classical
invariant of a Legendrian link.
The \emph{Thurston-Bennequin} and \emph{rotation numbers} of a null-homologous link $L$ in a rational homology contact 3-sphere $(M,\xi)$
are defined as follows. The first one, denoted with $\text{tb}(L)$, is the linking number of the contact framing $\xi_L$
of $L$ respect to $\xi$ and a Seifert surface for $L$.
While, the second one, denoted with $\text{rot}(L)$, 
is defined as the numerical obstruction to extending a non-zero vector field, everywhere 
tangent to the knot, to a Seifert surface of $K$ (see \cite{Etnyre4}).

If in $(M,\xi)$ there is an embedded disk $E$, with Legendrian boundary, such that 
$\text{tb}(\partial E)=0$ then the structure 
$\xi$ is \emph{overtwisted}; otherwise it is called \emph{tight}. 
Furthermore, in overtwisted structures there are two types of Legendrian and transverse  
links: \emph{non-loose} if their complement is tight and
\emph{loose} if it is overtwisted. 
We have the following proposition.
\begin{prop}
 \label{prop:loose}
 Let $L$ be a loose Legendrian link in an overtwisted contact $3$-manifold $(M,\xi)$. Then we have that 
 $\mathfrak L(L,M,\xi)=[0]$.
\end{prop}
We use the invariant $\mathfrak L$ to prove the existence of non-loose Legendrian $n$-component links, with loose 
components, in many 
overtwisted contact 3-manifolds.
\begin{teo}
 \label{teo:nonloose}
 Suppose that $(M,\xi)$ is an overtwisted $3$-manifold such that there exists a contact structure $\zeta$ with 
 $\mathfrak t_{\xi}=\mathfrak t_{\zeta}$ and $\widehat c(M,\zeta)\neq[0]$.
 Then there is a non-split Legendrian $n$-component
 link $L$, for every $n\geq 1$, such that $\mathfrak L(L,M,\xi)$ is non-zero and all of its sub-links are loose. 
 In particular, $L$ is non-loose and stays non-loose after a sequence of negative stabilizations.
 
 Futhermore, the transverse link $T_L$, obtained as transverse push-off of $L$, is again non-split and $\mathfrak T(T_L,M,\xi):=\mathfrak L(L,M,\xi)$
 is non-zero, which means that $T_L$ is also non-loose. 
\end{teo}
Furthermore, we give examples of \emph{non-loose, non-simple Legendrian and transverse link types}.
Here non-simple means that there 
exists a pair of Legendrian (transverse)
links that have the same classical invariants and their components are Legendrian (transverse)
isotopic knots, but they are not Legendrian (transverse) isotopic (as links). Moreover, with non-loose we mean that such a 
pair consists of two non-loose Legendrian (transverse) links.
\begin{teo}
 \label{teo:last}
 Suppose that $(M,\xi)$ is an overtwisted $3$-manifold as in Theorem \ref{teo:nonloose}.
 Then in $(M,\xi)$ there is a pair of non-loose,
 non-split Legendrian (transverse) $n$-component links, with the same classical invariants and Legendrian
 (transversely) isotopic components, but that are not Legendrian (transversely) isotopic.
\end{teo}
This paper is organized as follows. In Section \ref{section:two} we define open book decompositions and we describe the
compatibility condition with Legendrian links. Moreover, we show that such adapted open books always exist.
In Section \ref{section:three} we give some remarks on link Floer homology. In Section \ref{section:four} we define 
$\mathfrak L$ and we prove its invariance under Legendrian isotopy. In Section \ref{section:five} we show some properties
of the invariant $\mathfrak L$, including the relations with the contact invariant $\widehat c$ and the classical 
invariants of Legendrian links. In Section \ref{section:six} we generalize the transverse invariant $\mathfrak T$ to links
and we describe some of its properties. Finally, in Section \ref{section:seven} we give some applications of our invariants.

\paragraph*{Acknowledgements:}
The author would like to thank Andr\'as Stipsicz for all the conversations about Legendrian invariants and Heegaard Floer
homology, which were vital for the birth of this paper, Stefan Friedl for helping to formalize the statement of Theorem \ref{teo:main}, and Irena Matkovi\v{c} for her help in understanding contact
topology.
In addition, the author wants to thank Peter Ozsv\'ath for giving me the opportunity to spend a semester at Princeton University,
where most of the content of this paper was developed. The author also thanks the anonymous referees for their suggestions. 

The author is supported by the ERC Grant LDTBud from the
Alfr\'ed R\'enyi Institute of Mathematics and a Full Tuition Waiver for a Doctoral program
at Central European University.

\section{Open book decompositions adapted to a Legendrian link}
\label{section:two}
\subsection{Definition}
\label{subsection:reduced}
Let us start with a contact 3-manifold $(M,\xi)$. We say that
an \emph{open book decomposition} for $M$ is a pair $(B,\pi)$ where
\begin{itemize}
   \item the binding $B$ is a smooth link in $M$;
   \item the map $\pi:M\setminus B\rightarrow S^1$ is a locally trivial fibration such that 
       $\overline{\pi^{-1}(\theta)}=S_{\theta}$ is a compact surface with 
       $\partial S_{\theta}=B$ for every $\theta\in S^1$. The surfaces $S_{\theta}$ are called \emph{pages} of the open book. 
\end{itemize}
Moreover, the pair $(B,\pi)$ supports $\xi$ if, up to isotopy, we can find a 1-form $\alpha$ for 
$\xi$ such that 
\begin{itemize}
  \item the 2-form $\dd\alpha$ is a positive area form on each page $S_{\theta}$;
  \item we have $\alpha>0$ on the binding $B$. 
\end{itemize}  
Let us take the page $S_1=\overline{\pi^{-1}(1)}$, which is an oriented, connected, compact surface of genus $g$ and with 
$l$ boundary components. Assume from now on that links are always oriented; suppose that an $n$-component Legendrian link $L$ in $(M,\xi)$ sits inside $S_1$ and its components
represent $n$ independent elements in $H_1(S_1;\F)$. 
We say that a collection of
disjoint, simple arcs $A=\{a_1,...,a_{2g+l+n-2}\}$ is a \emph{system of generators for $S_1$ adapted to $L$} if
the following conditions hold:
\begin{enumerate}
  \item the subset $A^1\sqcup A^2\subset A$, where $A^1=\{a_1,...,a_n\}$ and 
        $A^2=\{a_{n+1},...,a_{2g+l-1}\}$, is a basis of $H_1(S_1,\partial S_1;\F)\cong\F^{2g+l-1}$;
  \item we have that $L\pitchfork a_i=\{1\text{ pt}\}$ for every $a_i\in A^1$ and $L\cap a_i=\emptyset$ for $i>n$. 
        The arcs in the subset $A^1$ are called \emph{distinguished arcs};
  \item the subset $A^3=\{a_{2g+l},...,a_{2g+l+n-2}\}$ is such that 
        the surface $S$ given by the closure of $S_1\setminus A$ has $n$ disks as connected components, each one containing 
        exactly one component of $L$.
        
        The arcs in $A^2\sqcup A^3$ that appear twice on the same component of the boundary of $S$ are 
        called \emph{dead arcs}; the others \emph{living arcs};
  \item the elements in $A^3$ are all living arcs; moreover, they  
        separates a unique pair of components of the disk $S_1\setminus(A^1\sqcup A^2)$.
\end{enumerate}
With this definition in place we say that the triple $(B,\pi,A)$ is an open book decomposition \emph{adapted} to the Legendrian link 
$L$ if
\begin{itemize}
  \item the pair $(B,\pi)$ is compatible with $(M,\xi)$;
  \item the link $L$ is contained in the page $S_1$;
  \item the $n$ components of $L$ are independent in $S_1$;
  \item the set $A$ is a system of generators for $S_1$ adapted to $L$.
\end{itemize}
 In this case we
 also say that the adapted open book decomposition
 $(B,\pi,A)$ is compatible with the triple $(L,M,\xi)$. It is important to observe that, since the
 components of $L$ are required to be independent in homology, we only consider open book decompositions
 with pages not diffeomorphic to a disk.
 
\subsection{Existence}
 We need to prove that open book decompositions adapted to Legendrian links always exist. In order to do this we recall the 
 definition of contact cell decomposition
 (of a contact 3-manifold) and ribbon of a Legendrian graph. 
 A \emph{contact cell decomposition} of $(M,\xi)$ is a finite CW-decomposition of $M$ such that
 \begin{enumerate}
  \item the 1-skeleton is a connected Legendrian graph;
  \item each 2-cell $E$ satisfies $\text{tb}(\partial E)=-1$;
  \item the contact structure $\xi$ is tight when restricted to each 3-cell.
 \end{enumerate}
 Moreover, if we have a Legendrian link 
 $L\hookrightarrow (M,\xi)$ then we also suppose that 
 \begin{enumerate}
  \setcounter{enumi}{3}
  \item the 1-skeleton contains $L$.
 \end{enumerate}
Denote the 1-skeleton of a contact cell decomposition of $(M,\xi)$ with $G$. 
Then $G$ is a Legendrian graph and its \emph{ribbon} is a compact surface
$S_G$ satisfying:
\begin{itemize}
  \item $S_G$ retracts onto $G$;
  \item $T_pS_G=\xi_p$ for every $p\in G$;
  \item $T_pS_G\neq\xi_p$ for every $p\in S\setminus G$.
\end{itemize}
We say that an adapted open book decomposition $(B,\pi,A)$, compatible with the triple $(L,M,\xi)$, comes from a contact cell 
decomposition if $S=\overline{\pi^{-1}(1)}$ is a ribbon of the 1-skeleton of $(M,\xi)$.
\begin{proof}[Proof of Theorem \ref{teo:reduced}]
  Corollary 4.23 in \cite{Etnyre} assures us that we can always find an open book decomposition $(B,\pi)$ which comes from
  a contact cell decomposition of $(M,\xi)$
  \begin{figure}[ht]
  \centering
  \def\svgwidth{7cm}
  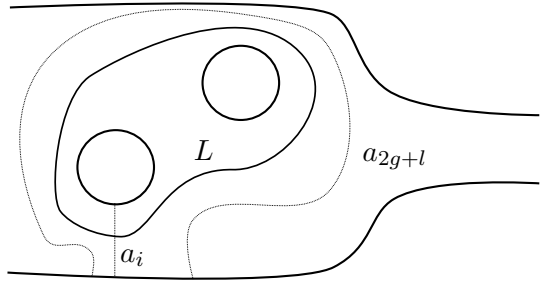   
  \caption{We add a new living arc which is parallel to $L_i$ except near the distinguished arc.}
  \label{Reduced}
  \end{figure}
  and the link $L$ is contained in $S_1$; where the page $S_1$ is precisely a ribbon of the 
  1-skeleton of $(M,\xi)$. The proof of this corollary also gives that the two framings of $L$ agree.
  
  The components of $L$ are independent
  because it is easy to see from the construction that there is a collection of disjoint, 
  properly embedded arcs $\{a_1,...,a_n\}$ in $S_1$ 
  such that 
  $$L_i\pitchfork a_i=\{1 \text{ pt}\}\:\:\:\:\text{ and }\:\:\:\:L_i\cap
  \left(\bigcup_{j\neq i} a_j\right)=\emptyset$$ for every $i$.   
  To conclude we only need to show that there exists a system of generators $A=\{a_1,...,a_{2g+l+n-2}\}$ for $S_1$
  which is adapted to $L$.  
  
  The arcs $a_1,...,a_n$ 
  are taken as before.
  If we complete $L$ to a basis of $H_1(S;\F)$ then Alexander duality gives a basis $\{a_1,...,a_{2g+l-1}\}$ with the same 
  property.
  We define $a_{2g+l},...,a_{2g+l+n-2}$ in the following way: each new living arc is parallel to $L_i$ and extended by following the distinguished arc until the boundary of $S_1$
  as in Figure \ref{Reduced}.
  Clearly, it disconnects the surface, because the first $2g+l-1$ arcs are already a basis of $H_1(S,\partial S;\F)$. If
  one of the components contains no 
  distinguished arcs, like in Figure \ref{Reduced1},
  then we choose the other endpoint of $a_i$ to extend the arc.
\end{proof}
In the following paper we use adapted open book decompositions to present Legendrian links in contact 3-manifolds.
 \begin{figure}[ht]
  \centering
  \def\svgwidth{7cm}
  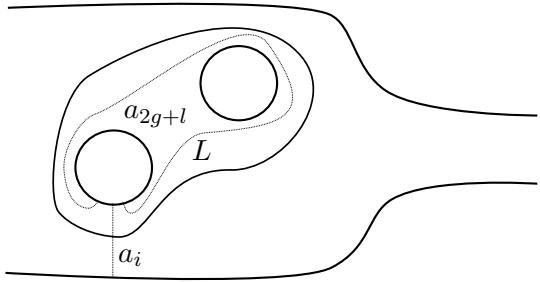   
  \caption{The picture appears similar to Figure \ref{Reduced}, but this time the 
                                                                     new arc follows the distinguished arc in the opposite direction.}
  \label{Reduced1}
  \end{figure}
Moreover, we study how to relate two different open book decompositions representing isotopic Legendrian links.

\subsection{Abstract open books}
\label{subsection:abstract}
An \emph{abstract open book} is a quintuple $(S,\Phi,\mathcal A,z,w)$ defined as follows.
We start with the pair $(S,\Phi)$. We have that $S=S_{g,l}$ is an oriented, connected, compact surface of genus $g$ and with 
$l$ boundary components, not diffeomorphic to a disk.
While $\Phi$ is the isotopy class of a diffeomorphism of $S$ into itself which is the identity on $\partial S$.
The class $\Phi$ is called \emph{monodromy}.
 
The pair $(S,\Phi)$ determines a contact 3-manifold up to contactomorphism.
The construction is described in \cite{Etnyre} Definition 2.3, Lemma 2.4
and Theorem 3.13.
 
The set $\mathcal A$ consists of two collections of properly embedded arcs, $B=\{b_1,...,b_{2g+l+n-2}\}$ and 
$C=\{c_1,...,c_{2g+l+n-2}\}$ in $S$ with $n\geq 1$,  
\begin{figure}[ht]
        \centering
        \def\svgwidth{7cm}
        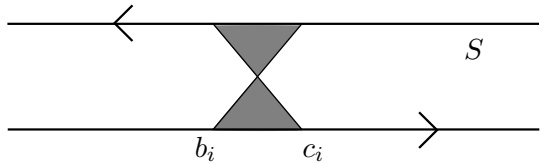   
        \caption{Two arcs in strip position.}
        \label{Strip}
\end{figure}
such that all the arcs in $B$ are disjoint, all the arcs in $C$ are disjoint and each pair $b_i,c_i$ appears as in 
Figure \ref{Strip}.
We suppose that each strip, the grey area between $b_i$ and $c_i$, is disjoint from the others. We also want $B$ and 
$C$ to represent two system of generators
for the relative homology group $H_1(S,\partial S;\F)$. In this way, if we name 
the strips $\mathcal A_i$, we have that
$S\setminus{\bigcup{b_i}}$, $S\setminus{\bigcup{c_i}}$ and $S\setminus{\bigcup{\mathcal A_i}}$ have exactly $n$
connected components.
 
Finally, $z$ and $w$ are two sets of basepoints: $w=\{w_1,...,w_n\}$ and $z=\{z_1,...,z_n\}$. 
We require these sets to have the following properties:
\begin{itemize}
  \item there is a $z_i$ in each component of $S\setminus\bigcup\mathcal A_i$, with the condition that every 
        component contains exactly one element of $z$; 
  \item each $w_i$ is 
        inside one of the strips $\mathcal A_i$, between the arcs $b_i$ and $c_i$, with the property that every strip contains 
        at most one element of $w$.
  \begin{figure}[ht]
        \centering
        \def\svgwidth{12cm}
        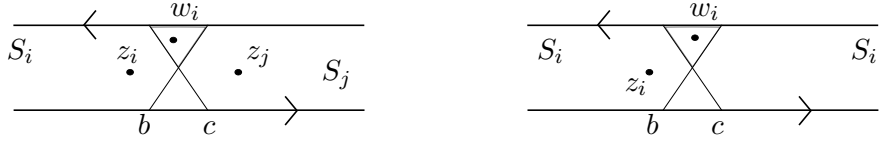   
        \caption{On the left $S_i$ and $S_j$ are different components of $S\setminus\bigcup\mathcal A_i$.}
        \label{Basepoints}
  \end{figure}
       See Figure \ref{Basepoints};
 \item we choose $z$ and $w$ in a way that each component of $S\setminus B$ and $S\setminus C$ contains exactly 
       one element of $z$ and one element of $w$.
\end{itemize}
We can draw an $n$-component link inside $S$ using the following procedure: we go from the $z$'s to the $w$'s by 
crossing $B$ and from $w$ to $z$ by crossing $C$, as shown in Figure \ref{Rule}. 
Moreover, we observe that the components of the link are independent in $S$. 
 
Using the Legendrian realization theorem (Theorem 2.7) in \cite{LOSS} and the procedure we described we can prove that
every abstract open book determines a Legendrian link in a contact 3-manifold up to contactomorphism. 
We are now interested in proving that an adapted open book decomposition $(B,\pi,A)$ always
determines an abstract open book.
\begin{figure}[ht]
 \centering
 \def\svgwidth{12cm}
 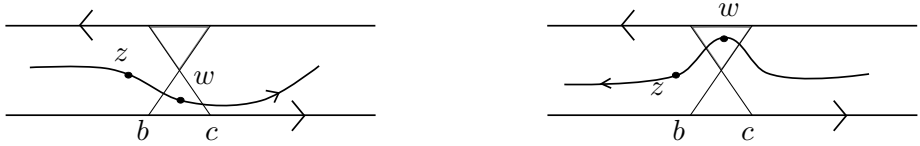   
 \caption{The link is oriented accordingly to the basepoints.}
 \label{Rule}
\end{figure}
\begin{prop}
  \label{prop:abstract}
  We can associate to 
  an adapted open book decomposition $(B,\pi,A)$, compatible with the triple $(L,M,\xi)$, an abstact open book
  $(S,\Phi,\mathcal A,z,w)$ up to isotopy.
\end{prop}
\begin{proof}
  The surface $S$ is obviously the page $\overline{\pi^{-1}(1)}$. Now
  consider the subsets of unit complex numbers $I_{\pm}\subset S^1\subset\C$ with non-negative and non-positive imaginary part.
  Since they are contractible, we have that
  $\pi\lvert_{\pi^{-1}(I_{\pm})}$ is a trivial 
  bundle. This gives two
  diffeomorphisms between the pages $S_1$ and $S_{-1}$. The monodromy $\Phi$ is precisely the isotopy class of 
  the composition of these diffeomorphisms.
 
  At this point, we want to define the strips $\mathcal A$. Hence, we need the collections of arcs $B$ and $C$: starting from
  the system of generators $A$, which is adapted to $L\subset S$, we take them to be both isotopic to $A$, in
  ``strip position'' like in Figure \ref{Strip} and such that $L$ does not cross the intersections of the arcs in $B$
  with the ones in $C$. 
  We only have an ambiguity on, following the orientation of $L$, which is the first arc intersected by $L$. To solve this
  problem we have to follow the rule that we fixed in Figure \ref{Rule}.
  
  Now we need to fix the basepoints. We put the $z$'s on $L$; exactly one on each
  component of $L\setminus(L\cap\mathcal A)$. The points in $z$ on different components of $L$ stay in different domains
  because of Condition 3 in the definition of adapted system of generators.
  Then $S\setminus\mathcal A$ has $n$ connected components, since the components of $L$ are independent, and each of these 
  contains exactly one element of $z$.
  Since the $z$'s are outside of the strips then we have that each component of $S\setminus B$ and $S\setminus C$ also 
  contains only one $z_i$.   
 
  The $w$'s are still put on $L$, but inside the strips containing the $n$ distinguished arcs. The points $w_1,...,w_n$ 
  correspond to $\mathcal A_1,...,\mathcal A_n$.
\end{proof}

\section{Heegaard Floer homology}
\label{section:three}
\subsection{Legendrian Heegaard diagrams}
\label{subsection:holomorphic}
 Heegaard Floer homology has been introduced by Ozsv\'ath and Szab\'o in \cite{Ozsvath1}; later it was
 generalized to knots and links in \cite{Ozsvath2} and independently
 by Rasmussen, in his PhD thesis \cite{Rasmussen}. In its original formulation links and 3-manifolds were presented with
 Heegaard diagrams, but in this work we only use a specific type of these diagrams, obtained from adapted open book
 decompositions, that we call Legendrian Heegaard diagrams.
 
 From now on a 3-manifold $M$ is always supposed to be a rational homology sphere. Given an adapted open book decomposition
 $(B,\pi,A)$, compatible with the triple $(L,M,\xi)$, a \emph{Legendrian Heegaard diagram} consists of a quintuple  
 $(\Sigma,\alpha,\beta,w,z)$ where $\Sigma$ is a closed, oriented surface, $\alpha$ and $\beta$ are two collections 
 of curves in $\Sigma$ and $w$ and $z$ are two sets of $n$ basepoints in $\Sigma$.
 
 The surface $\Sigma$ is $S_1\cup\overline{S_{-1}}$, where $S_{\pm 1}=\overline{\pi^{-1}(\pm 1)}$; since $\pi$ is a locally
 trivial fibration the pages $S_1$ and $S_{-1}$ are diffeomorphic, but we reverse the orientation of the second one when
 we glue them together.
 
 We have that $A=\{a_1,...,a_{2g+l+n-2}\}$ and we choose $B=\{b_1,...,b_{2g+l+n-2}\}$ in a way that $A$ and $B$ are 
 like in Figure \ref{Strip}. We recall that $g$ is the genus of $S_1$, $l$ is the number of boundary components
 of $S_1$ and $n$ is the number of components of $L$.
 \begin{figure}[ht]
 \centering
 \def\svgwidth{12cm}
 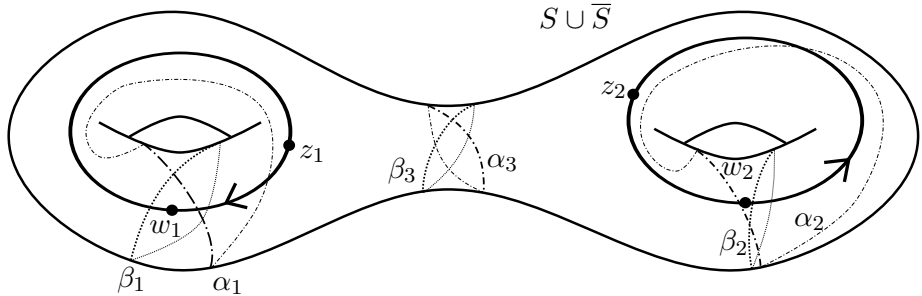   
 \caption{A diagram for the standard Legendrian 2-unlink in $(S^3,\xi_{\text{st}})$.} 
 \label{Unlink}
 \end{figure}
 Then we define $\alpha_i=b_i\cup(h\circ\Phi)(b_i)$ and $\beta_i=a_i\cup h(a_i)$ for every $i=1,...,2g+l+n-2$,
 where
 $h:S_1\rightarrow\overline{S_{-1}}$ is the Identity, and $\Phi$ is the monodromy, which is fixed
 by the open book as seen in Proposition \ref{prop:abstract}. Finally, 
 the $z$'s and the $w$'s are the set of baspoints that we introduced in Subsection \ref{subsection:abstract}.

 In the settings of \cite{Ozsvath1} and \cite{Ozsvath2} the Legendrian Heegaard diagram $(\Sigma,\alpha,\beta,w,z)$ is a 
 diagram for the (smooth) link $L$ in the 3-manifold $\overline M$ (\cite{Ozsvath} Section 3), that is $M$ considered with the
 opposite orientation.
 We remark that here $\alpha$ and $\beta$ are swapped respect to the papers of Ozsv\'ath and Szab\'o.
 
 We observe that, given $(B,\pi,A)$, the only freedom in the choice of the Legendrian Heegaard diagram is in the arcs
 $(h\circ\Phi)(b_1),...,(h\circ\Phi)(b_{2g+l+n-2})$ inside $S_{-1}$, which depend on the isotopy class of 
 $\Phi(b_1),...\Phi(b_{2g+l+n-2})$.
 
\subsection{Basics of Heegaard Floer theory}
 Let us consider a Legendrian Heegaard diagram $D=(\Sigma,\alpha,\beta,w,z)$, coming from an adapted open book 
 decomposition compatible with
 $(L,M,\xi)$. Consider the $(2g+l+n-2)$-dimensional tori \[T_{\alpha}=\alpha_1\times...\times\alpha_{2g+l+n-2}\quad\text{ and }\quad\T_{\beta}=\beta_1\times...\times\beta_{2g+l+n-2}\] in the symmetric power $\text{Sym}^{2g+l+n-2}(\Sigma)$ of $\Sigma$
 and define $\widehat{CF}(D)$ and $cCFL^-(D)$ respectively as the $\F$-vector space and the free $\F[U]$-module generated
 by the elements of the transverse intersection $\T_{\alpha}\cap\T_{\beta}$.
 
 Fix an appropriate symplectic and compatible almost-complex structure
 $(\omega,J)$ on $\text{Sym}^{2g+l+n-2}(\Sigma)$. For every relative
 homology class $\phi\in\pi_2(x,y)$ we define $\mathfrak M(\phi)$ as the moduli space of 
 $J$-holomorphic maps from the unit disk $D\subset\C$ to
 $\left(\text{Sym}^{2g+l+n-2}(\Sigma),J\right)$ with the appropriate boundary conditions, see \cite{Ozsvath1}.
 The formal dimension of $\mathfrak M(\phi)$,
 denoted by $\mu(\phi)$, is the Maslov index; moreover, we call $\widehat{\mathfrak M}(\phi)$ the
 quotient $\mathfrak M(\phi)/\R$ given by translation. 
 
 Since we are working on 3-manifolds, we use the definition of $\text{Spin}^{c}$ structure given by 
 Turaev in \cite{Turaev}: an isotopy class,
 away from a point, of nowhere vanishing vector fields on the manifold.
 As described in \cite{Ozsvath} Section 3.3, we have two well-defined maps
 $$\mathfrak s_w,s_z:\T_{\alpha}\cap\T_{\beta}\longrightarrow\text{Spin}^{c}\left(\overline M\right)\cong H^2(M;\Z)$$ 
 which send an intersection 
 point $x\in\T_{\alpha}\cap\T_{\beta}$
 into the $\text{Spin}^{c}$ structures $\mathfrak s_w(x)$ and $\mathfrak s_z(x)$ on $\overline M$. 
 These maps are obtained by associating, to every $x\in\T_{\alpha}\cap\T_{\beta}$, two global 2-plane 
 fields $\mathfrak\pi_w(x)$ and $\mathfrak\pi_z(x)$
 on $\overline M$, whose restrictions are the corresponding $\text{Spin}^{c}$ structures.
 
 Since our manifold $M$ already comes with a contact structure $\xi$, we have that both $M$ and $\overline M$ are equipped 
 with a specific $\text{Spin}^{c}$
 structure, induced by $\xi$, that we denote with $\mathfrak t_{\xi}$.
 The elements of $\T_{\alpha}\cap\T_{\beta}$ can be partitioned according to the $\text{Spin}^c$ structures on $\overline M$,
 resulting in decompositions 
 $$\widehat{CF}(D)=\bigoplus_{\mathfrak t\in\text{Spin}^{c}\left(\overline M\right)}\widehat{CF}(D,\mathfrak t)
 \:\:\:\:\:\:\:\:\:\:\text{ and }\:\:\:\:\:\:\:\:\:\:
 cCFL^-(D)=\bigoplus_{\mathfrak t\in\text{Spin}^{c}\left(\overline M\right)}cCFL^-(D,\mathfrak t)\:;$$ where
 $\widehat{CF}(D,\mathfrak t)$ is generated by the intersection points $x\in\T_{\alpha}\cap\T_{\beta}$ such that
 $s_z(x)=\mathfrak t$, while $cCFL^-(D,\mathfrak t)$ by the ones such that $s_w(x)=\mathfrak t$. Note that the $\text{Spin}^c$
 structures in the two splittings may be different.
 
 For every $\phi\in\pi_2(x,y)$ we call \[n_{z_i}(\phi)=\#\left|\phi\cap\{z_i\}\times\text{Sym}^{2g+l+n-3}(\Sigma)\right|\quad\text{ and }\quad n_{w_i}(\phi)=\#\left|\phi\cap\{w_i\}\times\text{Sym}^{2g+l+n-3}(\Sigma)\right|\]where here we mean algebraic 
 intersection. Moreover, we have 
 $$n_z(\phi)=\sum_{i=1}^nn_{z_i}(\phi)\:\:\:\:\:\:\:\:\text{ and }\:\:\:\:\:\:\:\:n_w(\phi)=\sum_{i=1}^nn_{w_i}(\phi)\:.$$
 We define the differential 
 $\widehat\partial_z:\widehat{CF}(D,\mathfrak t)\rightarrow\widehat{CF}(D,\mathfrak t)$ as follows:
 $$\widehat\partial_zx=\sum_{y\in\T_{\alpha}\cap\T_{\beta}\lvert_{\mathfrak t}}\sum_{\substack{\phi\in\pi_2(x,y),\\ 
 \mu(\phi)=1,n_z(\phi)=0}}
 \#\left|\widehat{\mathfrak M}(\phi)\right|\cdot y$$ for every $x\in\T_{\alpha}\cap\T_{\beta}\lvert_{\mathfrak t}$.
 We note that, since we are interested in $\phi$'s that are image of some $J$-holomorphic disks, we have that 
 $n_{z_i}(\phi),n_{w_i}(\phi)\geq 0$ (Lemma 3.2 in 
 \cite{Ozsvath1}). This means that $n_z(\phi)=n_w(\phi)=0$ if and only if $n_{z_i}(\phi)=n_{w_i}(\phi)=0$ 
 for every $i=1,...,n$. 
 
 The map $\widehat\partial_z$ is well-defined if the diagram $D$ is admissible, which means that
 every $\phi\in\pi_2(x,x)$ with $n_w(\phi)=0$, representing a non-trivial homology class, has both positive and negative local 
 multiplicities.
 Usually we have a distinction between weak and strong admissibility,
 but if $M$ is a rational homology 3-sphere then the weakly and strongly admissible conditions coincide; 
 see Remark 4.11 in \cite{Ozsvath1} and Definition 3.5 and Subsection 4.1 in \cite{Ozsvath}. Given a diagram we can 
 always achieve admissibility with isotopies.
 \begin{prop}
  \label{prop:admissibility}
  Suppose $D=(\Sigma,\alpha,\beta,w,z)$ is a Legendrian Heegaard diagram 
  given by an adapted
  open book decomposition compatible with $(L,M,\xi)$,
  where $M$ is a rational homology $3$-sphere. Then $D$
  is always admissible up to isotope the arcs in $\overline B$.
 \end{prop}
 \begin{proof}
  It follows from Proposition 3.6 in \cite{Ozsvath} and Theorem 2.1 in \cite{Plamenevskaya}.
 \end{proof}
 The fact that $\widehat\partial_z\circ\widehat\partial_z=0$ is proved in \cite{Ozsvath1}. This gives that
 the pair $\left(\widehat{CF}(D,\mathfrak t),\widehat\partial_z\right)$ is a chain complex.
 In the definition of $\left(\widehat{CF}(D,\mathfrak t),\widehat\partial_z\right)$ we never use the basepoints in $w$, in fact
 the complex does not depend on the link $L$, but only on the number of its components. Moreover, in \cite{Ozsvath1} is proved
 that the homology $\widehat{HF}(D,\mathfrak t)$ of $\left(\widehat{CF}(D,\mathfrak t),\widehat\partial_z\right)$
 is an invariant of the $\text{Spin}^c$ 3-manifold $\left(\overline M,\mathfrak t\right)$ if the number of basepoints
 in $D$ is fixed. When $n=1$ the homology group
 is usually denoted with $\widehat{HF}\left(\overline M,\mathfrak t\right)$.
 
 In addition, since $M$ is a rational homology sphere, $\widehat{CF}(D,\mathfrak t)$ comes with an additional $\Q$-grading,
 called \emph{Maslov grading} \cite{Ozsvath4},
 given by $M_z(x)=d_3(\pi_z(x))$; where $d_3$ is the Hopf invariant of a nowhere zero vector 
 field. The differential $\widehat\partial_z$ drops the Maslov grading by one and then we have that
 $$\widehat{HF}(D,\mathfrak t)=\bigoplus_{d\in\Q}\widehat{HF}_d(D,\mathfrak t)\:;$$ where each $\widehat{HF}_d(D,\mathfrak t)$
 is a finite dimensional $\F$-vector space. 
 
\subsection{Link Floer homology}
 We consider the free $\F[U]$-module
 $cCFL^-(D,\mathfrak t)$ and we define a new differential 
 $$\partial^-:cCFL^-(D,\mathfrak t)\rightarrow cCFL^-(D,\mathfrak t)$$
 in the following way:
 $$\partial^-x=\sum_{y\in\T_{\alpha}\cap\T_{\beta}\lvert_{\mathfrak t}}\sum_{\substack{\phi\in\pi_2(x,y),\\ 
 \mu(\phi)=1,n_z(\phi)=0}}
 \#\left|\widehat{\mathfrak M}(\phi)\right|\cdot U^{n_w(\phi)}y$$  
 and $$\partial^-(Ux)=U\cdot\partial^-x$$ for every $x\in\T_{\alpha}\cap\T_{\beta}\lvert_{\mathfrak t}$.
 If $D$ is admissible then $\partial^-$ is also well-defined. Moreover, the fact that $\partial^-\circ\partial^-=0$ is proved 
 in Lemma 4.3 in \cite{Ozsvath}. Hence, the pair $\left(cCFL^-(D,\mathfrak t),\partial^-\right)$ is a chain complex.
 From \cite{Ozsvath} we know that the homology of $\left(cCFL^-(D,\mathfrak t),\partial^-\right)$, that 
 is denoted with $cHFL^-\left(\overline M,L,\mathfrak t\right)$, is a smooth isotopy invariant of $L$ in $\overline M$. 

 As before we can define the Maslov grading as $M(x)=M_w(x)=d_3(\mathfrak\pi_w(x))$ for every intersection point and we extend 
 it by taking $M(Ux)=M(x)-2$. Note that this definition of Maslov grading is different from the one used in the previous 
 subsection; in fact, now the set $w$ appears in place of $z$. In order to avoid confusion, we denote the Maslov
 grading of $x$ with $M(x)$ only in the case of links; otherwise we specify which set of basepoints is used in the definition. 
 Again the Maslov grading gives an $\F$-splitting
 $$cHFL^-\left(\overline M,L,\mathfrak t\right)=
 \bigoplus_{d\in\Q}cHFL^-_d\left(\overline M,L,\mathfrak t\right)\:;$$ where
 $cHFL^-\left(\overline M,L,\mathfrak t\right)$ is a finitely generated $\F[U]$-module and 
 each $cHFL^-_d\left(\overline M,L,\mathfrak t\right)$
 is a finite dimensional $\F$-vector space. 

 In the case of null-homologous links we can also define a $\frac{\Z}{2}$-grading on $cCFL^-(D,\mathfrak t)$ called \emph{Alexander grading}
 and denoted with $A$. Let us call $\overline M_L$ the 3-manifold with boundary 
 $\overline M\setminus\mathring{\nu(L)}$. Since $L$ has $n$ components, we have that $\partial\overline M_L$ consists of $n$ 
 disjoint tori. On this kind 
 of 3-manifold we define a relative $\text{Spin}^c$ structure as in \cite{Ozsvath}: the isotopy class, away from a point, of 
 a nowhere vanishing vector field such that
 the restriction on each boundary torus coincides with the canonical one (see \cite{Turaev}). We denote the set of the relative
 $\text{Spin}^c$ structures on
 $\overline M_L$ by $\text{Spin}^c\left(\overline M,L\right)$; then we have an identification of 
 $\text{Spin}^c\left(\overline M,L\right)$
 with the relative cohomology group $H^2\left(M_L,\partial\nu(L);\Z\right)$.
 Moreover, from \cite{Ozsvath} we have the following map
 $$\mathfrak s_{w,z}:\T_{\alpha}\cap\T_{\beta}\longrightarrow\text{Spin}^{c}\left(\overline M,L\right)\:.$$
 Clearly, the relative $\text{Spin}^c$ structure $\mathfrak s_{w,z}(x)$ extends to the actual $\text{Spin}^c$ structure
 $\mathfrak s_{w}(x)$. 
 Poincar\'e duality gives that 
 $$H^2\left(M_L,\partial\nu(L);\Z\right)\cong H_1\left(M_L;\Z\right)
 \cong H_1\left(L;\Z\right)\oplus H_1\left(M;\Z\right)\cong\Z^n\oplus H^2\left(M;\Z\right)\:;$$
 where we recall that $H^2\left(M;\Z\right)$ is a finite group. A basis of 
 the $\Z^n$ summand is given by the cohomology classes 
 $\{\text{PD}[\mu_i]\}_{i=1,...,n}$; where $\mu_i$
 is a meridian of the $i$-th component of $L$, oriented coherently. Then we have that 
 $$\mathfrak s_{w,z}(x)=\sum_{i=1}^n2s_i\cdot\text{PD}[\mu_i]+\mathfrak s_w(x)\:,$$ where each $s_i$ is an integer.
 Since $L$ admits a Seifert surface $F$,
 we define the Alexander absolute grading as follows:
 $$A(x)=\sum_{i=1}^ns_i=\dfrac{\mathfrak s_{w,z}(x)[F]}{2}\:,$$
 extended on the whole $cCFL^-(D,\mathfrak t)$ by saying that $A(Ux)=A(x)-1$. 

We have that $\partial^-$ preserves the 
Alexander grading and then there is another 
$\F$-splitting $$cHFL^-\left(\overline M,L,\mathfrak t\right)=
\bigoplus_{d,s\in\Q}cHFL^-_{d,s}\left(\overline M,L,\mathfrak t\right)\:;$$ where each 
$cHFL^-_{d,s}\left(\overline M,L,\mathfrak t\right)$ is a finite dimensional $\F$-vector space.

\section{The invariant}
\label{section:four}
\subsection{Special intersection points in Legendrian Heegaard diagrams}
\label{subsection:preliminaries}
In this subsection we define a cycle in the link Floer chain complex that we previously introduced. The corresponding 
homology class will be our Legendrian invariant.
Let us consider the only intersection point of $D=(\Sigma,\alpha,\beta,w,z)$ which lies in the page $S_1$.
We recall that, in general, the intersection points live in
the space $\T_{\alpha}\cap\T_{\beta}$, but they can be represented inside $\Sigma$. We denote this element with 
$\mathfrak L(D)$. 
\begin{prop}
 \label{prop:intersection}
 The intersection point $\mathfrak L(D)$ is such that $\partial^-\mathfrak L(D)=0$ and then 
 $\mathfrak L(D)$ is a cycle in $cCFL^-\left(D,\mathfrak t_{\mathfrak L(D)}\right)$; 
 where $\mathfrak t_{\mathfrak L(D)}$
 is the $\text{Spin}^c$ structure that it induces on $\overline M$.
\end{prop}
\begin{proof}
 Every $\phi\in\pi_2(\mathfrak L(D),y)$ such that $\mu(\phi)=1$, where $y\in\T_{\alpha}\cap\T_{\beta}$, has the property that 
 $n_z(\phi)>0$. The claim
 follows easily from the definition of the differential. 
\end{proof}
In Subsection \ref{subsection:proof} we show that more can be said on the $\text{Spin}^c$ structure $\mathfrak t_{\mathfrak L(D)}$. 
Now we spend a few words about the Ozsv\'ath-Szab\'o 
contact invariant $c(\xi)$, introduced in \cite{Ozsvath3}.
Given a Legendrian Heegaard diagram $D$;
let us call $c(D)$ the only intersection point which lies on the page $S_1$ as before, but now considered as an element in 
$\left(\widehat{CF}(D,\mathfrak t_{c(D)}),\widehat\partial_z\right)$.
The proof of Proposition \ref{prop:intersection} says that $c(D)$ is also a cycle. Moreover, we have the following theorem.
\begin{teo}[Ozsv\'ath and Szab\'o]
 \label{teo:contact}
 Let us consider a Legendrian Heegaard diagram $D$ with a single basepoint $z$,
 given by an open book compatible with a pair $(M,\xi)$, where $M$ 
 is a rational homology 3-sphere. Let us
 take the cycle $c(D)\in\widehat{CF}(D,\mathfrak t_{c(D)})$.
 
 Then the equivalence class of 
 $(\widehat{HF}\left(\overline M,\mathfrak t_{\xi}\right),[c(D)])$ is a contact invariant 
 of $(M,\xi)$ and we denote it with $\widehat c(M,\xi)$.
 Furthermore, we have the following properties:
 \begin{itemize}
  \item the $\text{Spin}^c$ structure $\mathfrak t_{c(D)}$ coincides with $\mathfrak t_{\xi}$;
  \item the Maslov grading of $c(D)$ is given by $M_z(c(D))=-d_3(M,\xi)$. 
 \end{itemize} 
\end{teo}
The proof of this theorem comes from \cite{Ozsvath3}, where Ozsv\'ath and Szab\'o first define the invariant 
$\widehat c(M,\xi)$, 
and \cite{Matic}, where Honda, Kazez and Mati\'c
give the reformulation using open book decompositions that we use in this paper.  

We want to prove a result similar to Theorem \ref{teo:contact}, but for Legendrian links. In other words, we show
that the isomorphism class of the element $[\mathfrak L(D)]$ inside the 
homology group
$cHFL^-\left(\overline M,L,\mathfrak t_{\mathfrak L(D)}\right)$ can be considered, as we are going to explain, a Legendrian invariant of the triple $(L,M,\xi)$, where here it is helpful to remember the observations in Remark \ref{remark:1}. 
Let us be more specific: we consider
a Legendrian link $L\hookrightarrow(M,\xi)$ in a rational homology contact 3-sphere; we associate two open book 
decompositions compatible with $(L,M,\xi)$, say
$(B_1,\pi_1,A_1)$ and $(B_2,\pi_2,A_2)$, and these determine (up to isotopy) two Legendrian Heegaard diagrams, that we call 
$D_1=(\Sigma_1,\alpha_1,\beta_1,w_1,z_1)$ and
$D_2=(\Sigma_2,\alpha_2,\beta_2,w_2,z_2)$ respectively. Then we want to find a chain map 
\begin{equation}
 \label{map}
 \Psi_{(D_1,D_2)}:cCFL^-(D_1,\mathfrak t)\longrightarrow cCFL^-(D_2,\mathfrak t)\:,
\end{equation}
that induces an isomorphism in 
homology, preserves the bigrading and
it is such that 
$\Psi_{(D_1,D_2)}(\mathfrak L(D_1))=\mathfrak L(D_2)$; where 
$\mathfrak t=\mathfrak t_{\mathfrak L(D_1)}=\mathfrak t_{\mathfrak L(D_2)}\in\text{Spin}^c
\left(\overline M\right)$.

Our strategy is to study how Legendrian isotopic triples are related to one another, which means to define a finite set of 
moves between two adapted open book decompositions. Then find the maps 
$\Psi_{(D_1,D_2)}$ in these particular cases. 

\subsection{Open books adapted to Legendrian isotopic links}
We want to show that, given two Legendrian isotopic links $L_1,L_2\hookrightarrow(M,\xi)$, two open book decompositions 
$(B_i,\pi_i,A_i)$, compatible
with the triples $(L_i,M,\xi)$, are related by a finite sequence of moves.
\subsubsection{Global isotopies}
The first lemma follows easily.
\begin{lemma}
 Let us consider an adapted open book
 decomposition $(B_1,\pi_1,A_1)$, compatible with the triple $(L_1,M,\xi)$, 
 and suppose that there is a contact isotopy of 
 $(M,\xi)$, sending $L_1$ into $L_2$.
 
 Then the time-$1$ map of the isotopy is a diffeomorphism 
 $F:M\rightarrow M$ such that $\left(F(B_1),\pi_1\circ F^{-1},F(A_1)\right)$ is an adapted open book decomposition, compatible
 with $(L_2,M,\xi)$.
\end{lemma}
This lemma says that, up to global contact isotopies, we can consider $(B_i,\pi_i,A_i)$ to be both compatible with a 
triple $(L,M,\xi)$, where the link $L$ is 
Legendrian isotopic to $L_i$ for $i=1,2$. In other words, we can just study the relation between two open book decompositions 
compatible with a single triple
$(L,M,\xi)$.
\subsubsection{Positive stabilizations}
Let us start with a pair $(S,\Phi)$. A positive stabilization of $(S,\Phi)$ is the pair
$\left(\widetilde{S},\widetilde\Phi\right)$ 
obtained
in the following way:
\begin{itemize}
 \item he surface $\widetilde S$ is given by adding a 1-handle $H$ to $S$;
 \item the monodromy $\widetilde\Phi$ is isotopic to $\Phi'\circ D_{\gamma}$. The map $\Phi'$ concides with $\Phi$ on $S$ 
       and it is the Identity on $H$.
       While $D_{\gamma}$ is the right-handed Dehn twist along a curve $\gamma$; which intersects $S\cap H$ transversely 
       precisely in the attaching sphere of $H$.
       See Figure \ref{Stabilization}.
\end{itemize}
\begin{figure}[ht]
 \centering
 \def\svgwidth{11cm}
 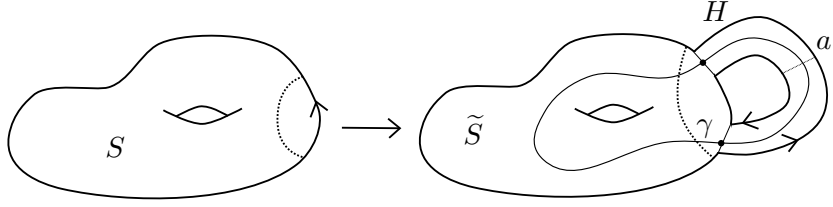   
 \caption{The arc $\gamma\cap\mathring S$ is a generic arc in the interior of $S$.}
 \label{Stabilization}
\end{figure}
We say that $(B',\pi',A')$ is a positive stabilization of $(B,\pi,A)$ if 
\begin{itemize}
 \item the pair $(S',\Phi')$, obtained from $(B',\pi')$, is a positive
       stabilization of $(S,\Phi)$, the one coming from $(B,\pi)$;
 \item the system of generators $A'$ is isotopic to $A\cup\{a\}$, where $a$ is the cocore of $H$ as in Figure 
       \ref{Stabilization}.
\end{itemize}
We recall the following theorem, proved by Giroux \cite{Giroux}. More details can be found in \cite{Etnyre} Section 3 and 4.
\begin{teo}
 \label{teo:Giroux}
 Two open book decompositions $(B_i,\pi_i,A_i)$ are compatible with contact isotopic triples $(L_i,M,\xi_i)$  
 if and only if they admit isotopic positive stabilizations.
\end{teo}
In other words, we may need to stabilize both open books many times and eventually we obtain other two open books 
$\left(B,\pi,\widetilde{A}_i\right)$, both compatible with
$L\hookrightarrow(M,\xi)$; which is contact isotopic to $(L_i,M,\xi_i)$ for $i=1,2$.
\subsubsection{Admissible arc slides}
Take an adapted system of generators $A$ for an $n$-component link $L$, lying inside a surface $S$. 
\begin{figure}[ht]
 \centering
 \def\svgwidth{12cm}
 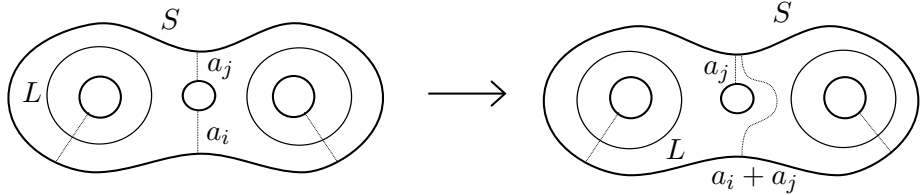   
 \caption{The arc $a_i+a_j$ is obtained by sliding $a_j$ over $a_i$.}
 \label{Admissible}
\end{figure}
We define admissible arc slide, 
a move that change $A=\{a_1,...,a_i,...,a_j,a_{2g+l+k-2}\}$ into $A'=\{...,a_i+a_j,...,a_j,...\}$; where $a_j$ is not a
distinguished arc and one of the endpoints of
$a_i$ and $a_j$ are adjacent, like in Figure \ref{Admissible}. We can prove the following proposition.
\begin{prop}
 \label{prop:admissible}
 Let us consider two open book decompositions $(B,\pi,A_i)$, compatible with the Legendrian link $L\hookrightarrow(M,\xi)$. 
 Then, after a finite number of admissible
 arc slides and isotopies on $A_i$, the open books coincide.
\end{prop}
We need three preliminary lemmas. We recall that, according to the definition in Subsection \ref{subsection:reduced}, we have $A=A^1\sqcup A^2\sqcup A^3$, where $A^1$ is the set of distinguished arcs, the set $A^3$ contains only living arcs and the arcs in $A^2$ can be either dead or alive. 
\begin{lemma}
 \label{lemma:inverse}
 An admissible arc slide, from $A$ to $A'$, can be inverted. In the sense that we can perform another admissible arc slide,
 now from $A'$ to $A''$, such that $A''$ is isotopic to $A$. 
\end{lemma}
\begin{proof}
 If the arc slide changes the arc $a_i$ into $a_i+a_j$ then it is easy to see that we can just slide $a_i+a_j$ over an 
 arc $a_j'$, isotopic to $a_j$;
 in a way that $a_i+a_j+a_j'$ is isotopic to $a_i$. 
\end{proof}
\begin{lemma}
 \label{lemma:three}
 Suppose that $a_i\in A^2\subset A$ is a living arc in an adapted system of generators. Then we can permute the arcs in $A$ in a way that $a_i\in A^3$ and $A$ is still an adapted system of generators.
\end{lemma}
\begin{proof}
 From Condition 4 in the definition of adapted system of generators, we have that $S\setminus A$ is the disjoint union of $n$ disks $D_1,...,D_n$ and each arc in $A^3$ connects exactly one pair of them.
 This means that if we build a graph $\mathcal G$ by taking the $D_i$'s as vertices and the arcs in $A^3$, which are all alive by definition, as edges then $\mathcal G$ is a tree. 
 
 Since $a_i$ is a living arc, we obtain that $a_i$ appears on the boundary of two distinct disks $D_1$ and $D_2$. Consider the minimal path in $\mathcal G$ from $D_1$ to $D_2$, and denote with $D_3$ the first disk that we meet after $D_1$. We call $a_j$ the arc in $A^3$ which connects $D_1$ to $D_3$; we claim that swapping $a_i$ and $a_j$ in $A$ yields again an adapted system of generators $B$.
 
 This follows because $S\setminus(B^1\sqcup B^2)$ is a disk, which implies $B^1\sqcup B^2$ is a basis of $H_1(S,\partial S;\F)$, and $B$ disconnects $S$ into the same collection of disks as $A$ does.
\end{proof}
\begin{lemma}
 \label{lemma:four}
 Suppose that we perform an admissible arc slide that changes $a_i$ into $a_i+a_j$. Then the set $A'$, obtained from $A$ by sliding $a_i$ over $a_j$, is still an adapted system of generators, possibly after rearranging the arcs. Furthermore, we have the following facts: 
 \begin{enumerate}[a)]
  \item the arc $a_i$ is distinguished if and only if $a_i+a_j$ is a distinguished arc; 
  \item the arc $a_i$ is alive if and only if $a_i+a_j$ is a living arc; 
  \item the arc $a_i$ is dead if and only if $a_i+a_j$ is a dead arc. 
 \end{enumerate}
\end{lemma}
\begin{proof}
 \begin{enumerate}[a)]
  \item The arc $a_i+a_j$ represents the sum of the relative homology classes of $a_i$ and $a_j$. At this point, since $a_j$ cannot be distinguished from the definition of admissible arc slide, we just need basic linear algebra. The converse follows in the same way, applying Lemma \ref{lemma:inverse}.
  \item We use Lemma \ref{lemma:three} to assume that $a_i\in A^3$. 
        Now, we observe that $a_i+a_j$ disconnects $S\setminus(A^1\sqcup A^2)$ because $A^1\sqcup A^2$ is a basis of $H_1(S,\partial S;\F)$; moreover, when considering $S\setminus A'$ the arc $a_i+a_j$ also appears on the boundary of two distinct disks.
        
        This argument proves both that $A'$ is still an adapted system of generators and $a_i+a_j$ is alive.
        The other implication follows easily from Lemma \ref{lemma:inverse}. 
  \item It follows from $b)$ and $c)$, but after an additional observation: it is not obvious that if $a_i$ is dead and $a_j\in A^3$ then $(A')^1\sqcup(A')^2$ is still a basis of $H_1(S,\partial S;\F)$. To see this, we assume that it is not the case; then using linear algebra we obtain that swapping $a_i+a_j$ and $a_j$ in $A'$ yields an adapted system of generators. Hence, from Lemma \ref{lemma:inverse} we can slide $a_i+a_j$ over $a_j$ and get (a permutation of ) $A$ back, but now $a_i$ is alive because of Point b). This is a contradiction.
 \end{enumerate}
\end{proof}
As an example we explain in detail what happens in Figure \ref{Admissible}. Denote by $d_1$ and $d_2$ the two distinguished arcs; then we have $A=\{d_1,d_2,a_i,a_j\}$ and $A'=\{d_1,d_2,a_j,a_i+a_j\}$. In $A$ the arcs $a_i$ and $a_j$ are both alive, while $a_j$ dies in $A'$ (Lemma \ref{lemma:four} does not say anything about that); accordingly to Lemma \ref{lemma:four}, we needed to rearrange the ordering of the arcs to guarantee that $A'$ remains an adapted system of generators. Notice however that being dead or alive does not depend on the ordering of the arcs in $A$.

The strategy of the proof of Proposition \ref{prop:admissible} 
is, say $S_1,...,S_n$ and $S_1',...,S_n'$ are the components of $S$ minus the living arcs of 
$A_1$ and $A_2$ respectively, we modify all
the living arcs, in both $A_1$ and $A_2$, with admissible arc slides; in order for $S_i$ to concide with $S_i'$ for 
every $i=1,...,n$. Moreover, we also
want that each living arc in $A_1$ becomes isotopic to a living arc in $A_2$. 

At the end, each of the components of $S$ 
minus the living arcs will contain a unique component of $L$, a unique distinguished arc and $S_i$ will have the same number of dead arcs respect to $S_i'$; in particular, this means that $A_1$ and $A_2$ will also have the same number of living arcs. We then conclude 
applying $n$ times Proposition 3.2 in
\cite{LOSS}; which is the knot case of this proposition. 
\begin{proof}[Proof of Proposition \ref{prop:admissible}]
\begin{figure}[ht]
  \centering
  \def\svgwidth{12cm}
  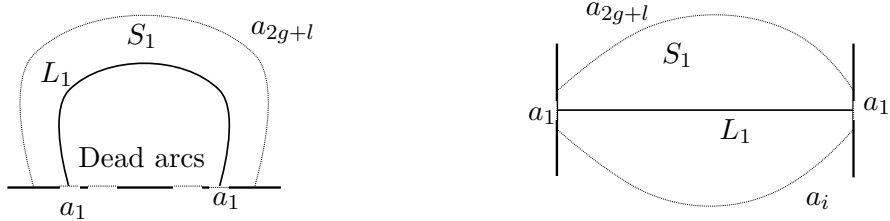   
  \caption{Case 1 is on the left. The figure actually shows a portion of $S\setminus A_1$.}
  \label{Cases}
 \end{figure}
 We want to prove that there is an adapted system of generators $A$ for $L$ in $S=\overline{\pi^{-1}(1)}$ such that
 $A$ is obtained, from $A_1$ and $A_2$, by 
 a sequence of admissible arc slides (and isotopies).
  
 We start from the component $S_1$. We can suppose that $\partial S_1$ contains the living arc 
 $a_{2g+l}\subset A_1^3\subset A_1$ and the distinguished arc
 $a_1\subset A_1^1\subset A_1$, with almost adjacent endpoints. 
 \begin{figure}[ht]
         \centering
         \def\svgwidth{8cm}
         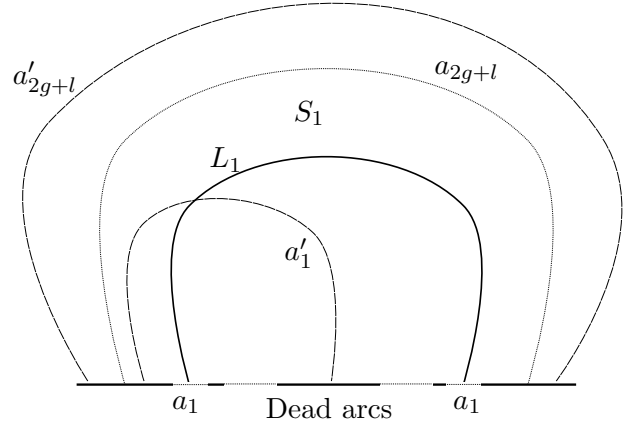   
         \caption{There are no other living arcs, except for $a_{2g+l}$ in $S_1$.}
         \label{Case1}
        \end{figure}
 We can have two cases: in the first one there are no other living arcs
 in the portion of $\partial S_1$ on the opposite side of $a_{2g+l}$. 
 In the second case they appear; possibly more than one, but we can suppose that there is
 exactly one of them. See Figure \ref{Cases}.
 \begin{enumerate}[$\text{Case}$ 1.]
  \item When we consider $S_1'$, which contains the same component of $L$ that is in $S_1$, after some arc slides we have 
        that it appears as in 
        Figure \ref{Cases} (left).
        This is because in the same figure we see that $S_1$ is split in two pieces by $L_1$ and the innermost one is not 
        connected in any way to other 
        components of $S$; in fact there are no living arcs on that side. This means that the same holds for $S_1'$ too. 
        At this point it is easy to see that 
        $S_1$ can be modified to be like in Figure \ref{Case1}; more explicitly, the living arcs are parallel and the 
        distinguished arcs lie in the region 
        where $L_1$ is.
  \item As before we have that also $S_1'$ appears like in the right part of Figure \ref{Cases} (always after some arc 
        slides). The reason is the same of previous case.
        Hence, now we can modify $S_1$ to be like in Figure \ref{Case2}; just in the same way as we did in Case 1. 
 \end{enumerate}
 We have obtained that the living arcs are fixed on $S_1$ and $S_1'$ and then the surfaces
 now have isotopic boundaries.
 \begin{figure}[ht]
         \centering
         \def\svgwidth{8cm}
         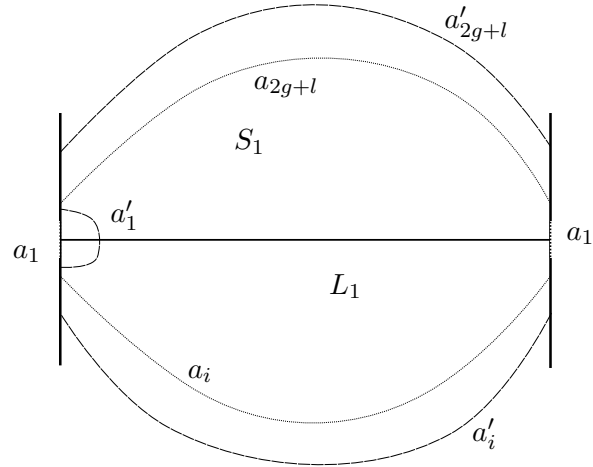   
         \caption{There are exactly two living arcs, namely $a_{2g+l}$ and $a_i$ in $S_1$.}
         \label{Case2}
        \end{figure}
 Hence, we can move 
 to another component $S_2$, whose boundary
 contains a living arc that has not yet been fixed, and we repeat the same process described before. 
 We may need to
 slide some living arcs over the ones in $A^3$ (or $(A')^3$) that
 we have already fixed in the previous step, but this is not a problem. We just iterate this procedure until all the $S_i$ 
 coincides with $S_i'$ and this
 completes the proof.
\end{proof}
The results of this subsection prove the following theorem.
\begin{teo}
 \label{teo:moves}
 If $L_1,L_2\hookrightarrow(M,\xi)$ are Legendrian isotopic links then the open book decompositions $(B_i,\pi_i,A_i)$, 
 which are 
 compatible with the corresponding triples, are related by a finite sequence of global contact isotopies, positive 
 stabilizations and admissible arc slides.
\end{teo}
Though they are easy to deal with, we do not have to forget that, when we define the corresponding Legendrian
Heegaard diagrams,
we need to consider the choises
of the monodromy $\Phi$ and the families of arcs $a$,$b$ and basepoints $z$,$w$ inside their isotopy classes.

\subsection{Invariance}
\subsubsection{Definition of the diagrams and global isotopies}
If two open book decompositions are related by a global isotopy then it easy to see that the induced abstract open book 
coincide, up to 
conjugation of the monodromy and isotopy of the curves 
and the basepoints in the diagrams. 

Hence, let us consider an abstract open book $(S,\Phi,A,z,w)$ and recall that in $S$ we have another set of arcs $B$,
as explained in Subsections \ref{subsection:abstract} and \ref{subsection:holomorphic}. The first check is easy: in fact if we 
perturb the basepoints inside
the corresponding components of $S\setminus A\cup B$ then even the complex $(cCFL^-(D,\mathfrak t),\partial^-)$ does not
change; where $D$ is a Legendrian Heegaard diagram obtained from $(S,\Phi,A,z,w)$. The same is true 
for an isotopy of $S$.

Now for what it concerns $S$ we are done, but when we define the chain complex we also consider the closed surface
$\Sigma$, obtained by gluing together $S$ and $\overline S$. 
\begin{figure}[ht]
 \centering
 \def\svgwidth{12cm}
 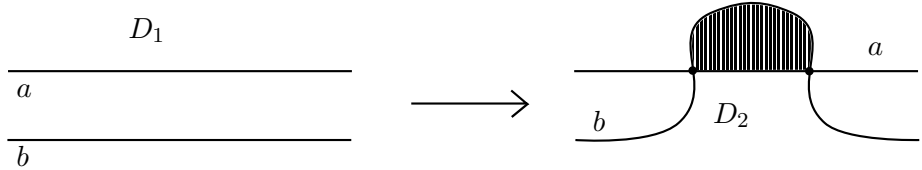   
 \caption{By moving the curve $b$ through a Hamiltonian isotopy we can introduce a pair of 
                                         canceling intersection points.} 
 \label{Isotopy}
\end{figure}
We still have some choices on $\overline S$, in fact by Proposition \ref{prop:admissibility} we may need
to modify the arcs $(h\circ\Phi)(b_i)$ (see Subsection \ref{subsection:holomorphic}) in their isotopy classes to achieve 
admissibility for $D=(\Sigma,\alpha,\beta,w,z)$.
Then the proof rests on the following proposition. 
\begin{prop}
 Suppose that two curves in a Heegaard diagram are related by the move shown in Figure \ref{Isotopy}. Then we can find a map 
 $\Psi_{(D_1,D_2)}$ as in 
 Equation \eqref{map}. 
\end{prop}
\begin{proof}
 The map $\Psi_{(D_1,D_2)}$ is constructed using a Hamiltonian diffeomorphism of the surface, as described in \cite{Ozsvath1} 
 Subsection 7.3. Since the new disks appear in  
 $\overline S$, we have that $\mathfrak L(D_1)$ is sent to $\mathfrak L(D_2)$ which lie both in $S$.
\end{proof}

\subsubsection{Admissible arc slides}
An arc slide $\{...,a_i,...,a_j,...\}\rightarrow\{...,a_i+a_j,...,a_j,...\}$ in $S$ corresponds to a handleslide 
$\{...,\alpha_i,...,\alpha_j,...\}\rightarrow\{...,\alpha_i',...,\alpha_j,...\}$ in $\Sigma$, where 
$\alpha_i'=a_i+a_j\cup h(a_i+a_j)\subset\Sigma$, again see Subsection \ref{subsection:holomorphic}.
Thus, a chain map $\Psi_{(D_1,D_2)}$, which induces an isomorphism in homology, is obtained by counting holomorphic triangles, 
as explained in \cite{Ozsvath}
Subsection 6.3 and Section 7. The admissibility of the arc slide is required only to avoid crossing a basepoint in $w$. 
Remember that for every arc slide we actually have two handleslides; in fact we need to slide both the $\alpha$ and the
$\beta$ curves.
The fact that $\Psi_{(D_1,D_2)}(\mathfrak L(D_1))=\mathfrak L(D_2)$ follows from Lemma 3.5 in \cite{Matic}; where the 
arc slides invariance is proved in 
the open books setting.

\subsubsection{Positive stabilizations}
At this point, in order to complete the proof of the invariance of $[\mathfrak L(D)]$, it would be enough to define 
$\Psi_{(D,D^+)}$, such that 
$\Psi_{(D,D^+)}(\mathfrak L(D))=\mathfrak L(D^+)$, in the case where $D$ and $D^+$ are obtained from an open book and one of 
its positive stabilizations.
Nevertheless, this is not what we prove. In fact, we define $L$-\emph{elementary positive stabilizations} as the ones such that the
curve $\gamma$, which is the curve 
in the page $S'$ that we used to perform the stabilization (see Figure \ref{Stabilization}), intersects the link $L$ 
(that sits in $S$ and then also in $S'$) in
at most one point transversely. Then what we actually prove is the existence of $\Psi_{(D,D^+)}$ for an $L$-elementary 
stabilization. To do this 
we only need the following theorem; which is a modification of Giroux's Theorem \ref{teo:Giroux} and whose
proof is explained in Section 4 in \cite{Etnyre}.
\begin{teo}
 \label{teo:Giroux_L}
 If $(B_i,\pi_i)$ are two open book decompositions, compatible with the triple $(L,M,\xi)$, then they admit isotopic 
 $L$-elementary stabilizations. Namely, there is
 another compatible open book $(B,\pi,A)$ which is isotopic to both $(B_1,\pi_1)$ and $(B_2,\pi_2)$, after an 
 appropriate 
 sequence of $L$-elementary stabilizations.
\end{teo}
Since we have already proved invariance under admissible arc slides, we can complete the open books $(B,\pi)$ and 
$(B^+,\pi^+)$
with every possible adapted system of generators and then
eventually define the map $\Psi_{(D,D^+)}$.
\begin{prop}
Let us consider the page $S=S_{g,l}=\overline{\pi^{-1}(1)}\supset L$ of $(B,\pi)$. Then we can always find $A$, an adapted system of 
generators for $L$ in $S$, with the property that $A$ is disjoint from the arc $\gamma'=\gamma\cap S$; where $\gamma$ is the 
curve that we used to perform the 
$L$-elementary stabilization. 
\end{prop}
\begin{proof}
 We have to study four cases:
 \begin{enumerate}[a)]
  \item the arc $\gamma'$ intersects $L$ (transversely in one point);
  \item the intersection $\gamma'\cap L$ is empty and $\gamma'$ does not disconnect $S$;
  \item the intersection $\gamma'\cap L$ is empty, the arc $\gamma'$ disconnects $S$ and $L$ is not contained in one of the two 
       resulting connected components of $S$;
  \item the intersection $\gamma'\cap L$ is empty, the arc $\gamma'$ disconnects $S$ and the link $L$ lies in one of the 
       resulting connected components of $S$. 
 \end{enumerate}
Let us start with Case a). We observe that $\gamma'$ does not disconnect $S$; in fact if this was the case then a component of 
$L$ would be split inside the two
resulting components of $S$, thus we would have at least another intersection point between $L$ and $\gamma'$; which is 
forbidden. We define $A=\{a_1,...,a_{2g+l+n-2}\}$
as follows. Take $a_1$ as a push-off of $\gamma'$; clearly $a_1$ is a distinguished arc, because it intersects $L_1$, a 
component of $L$, in one point.
Now call $a'$ the arc given by taking a push-off of $L_1$ and extend it 
through $a_1$, on one side of $L_1$; this is the same procedure described in the proof of Theorem \ref{teo:reduced}. If $a'$ 
disconnects $S$ into $S_1,S_2$ 
and there are components of $L$ that lie in both $S_i$
then we take $a'$ as a living arc; thus we extend $\{a_1,a'\}$ to an adapted system of generators $A$ by using 
Theorem \ref{teo:reduced}. On the other hand, if $L$ is contained in $S_2$ and $S_1$ is empty then we consider $\{a_1,a''\}$; 
where $a''$ is another push-off of $L_1$,
this time extended through $a_1$ on the other side of $L_1$. We still have problems when $a'$ does not disconnect $S$. We can 
fix this by taking
$\{a_1,a',a''\}$; which together disconnect $S$ into two connected components, one of them containing only $L_1$ and the other 
one $L\setminus L_1$. Again
we extend the set $\{a_1,a',a''\}$ to $A$ applying Theorem \ref{teo:reduced}.

The other three cases are easier. In Case b) we just denote with $a_{n+1}$ the push-off of $\gamma'$ and we complete it to $A$.

In Case c) we denote with $a_{2g+l}$ our push-off of $\gamma'$ and we take it as a living arc. Then we can complete it to
$A$.

Finally, Case d) is as follows. Since in this case the push-off is trivial in homology, and it bounds a surface disjoint from 
$L$, we can actually ignore it and easily 
find a set $A$ which never intersects $\gamma'$. 
\end{proof}
Now we have $(S,\Phi,A,z,w)$ the abstract open book obtained from $(B,\pi,A)$. Denote with $(S^+,\Phi^+,A^+,z,w)$ the 
one coming from $(B^+,\pi^+,A^+)$,
where $S^+=\overline{\left(\pi^+\right)^{-1}(1)}$ is $S$ 
with a 1-handle attached; 
$\Phi^+=\Phi'\circ D_{\gamma}$ where $\Phi'$ coincides with $\Phi$ on $S$, extended with the 
Identity on the new 1-handle; and 
$A^+=A\cup\{a\}$ with $a$ being the cocore of the new 1-handle (see Figure \ref{Stabilization}).
Then we call $D$ and $D^+$ the corresponding Legendrian Heegaard diagrams.

We define $\Psi_{(D,D^+)}$ in the following way. For every
$x\in\T_{\alpha}\cap\T_{\beta}\lvert_{\mathfrak t_{\mathfrak L(D)}}$ one has $\Psi_{(D,D^+)}(x)=x'$; where
$x'=x\cup\{a\cap b\}$ with $b$ being the arc in strip position with $a$, as in Figure \ref{Strip}. It results that 
$\Psi_{(D,D^+)}$ is a chain map because the curve $\alpha=b\cup(h\circ\Phi)(b)$
only intersects the curve $\beta=a\cup h(a)$, and moreover one has $\alpha\cap\beta=\{1\text{ pt}\}$, 
since we choose $A^+$ in a way that every arc in it is disjoint from $\gamma'$. We have that 
$\Psi_{(D,D^+)}$ induces an isomorphism in homology, because it is an isomorphism also on the level of chain complexes,
and sends $\mathfrak L(D)$ into $\mathfrak L(D^+)$.

\subsection{Proof of the main theorem}
\label{subsection:proof}
We prove our main result which defines the Legendrian invariant $\mathfrak L$.
\begin{proof}[Proof of Theorem \ref{teo:main}]
 We proved that $\mathfrak L(D)$ is a cycle in Proposition \ref{prop:intersection}. 
 The invariance follows from the results obtained in this section: in fact, we proved that if $D_1$ and $D_2$ are Legendrian
 Heegaard diagrams, representing Legendrian isotopic links, then we have a chain map $\Psi_{(D_1,D_2)}$ that preserves
 the bigrading and the $\text{Spin}^c$ structure and sends $\mathfrak L(D_1)$ to $\mathfrak L(D_2)$.  
 
 Now to see which is the corresponding 
 $\text{Spin}^c$ structure we recall that
 $$\mathfrak s_w(\mathfrak L(D))-\mathfrak s_z(\mathfrak L(D))=\text{PD}[L]$$ from 
 Lemma 2.19 in \cite{Ozsvath1}. Then we have 
 $$\mathfrak t_{\mathfrak L(D)}=\mathfrak s_w(\mathfrak L(D))=\mathfrak s_z(\mathfrak L(D))+\text{PD}[L]=
 \mathfrak t_{c(D)}=
 \mathfrak t_{\xi}\:.$$
\end{proof}
We note that the invariant can be a $U$-torsion class in the group 
$cHFL^-\left(\overline M,L,\mathfrak t_{\xi}\right)$, which means that there is a $k\geq 0$ such that 
$U^k\cdot\mathfrak L(L,M,\xi)=[0]$. Moreover, the cycle $\mathfrak L(D)$ possesses a bigrading $(d,s)$; such bigrading
is induced on the invariant $\mathfrak L(L,M,\xi)$,
because all the maps $\Psi$ defined in this section preserves both the Maslov and the Alexander
grading.

\section{Properties of \texorpdfstring{$\mathfrak L$}{L} and connected sums}
\label{section:five}
\subsection{Multiplication by \texorpdfstring{$U+1$}{U+1} in the link Floer homology group}
Take a Legendrian Heegaard diagram $D$, obtained from an adapted 
open book decomposition compatible with the triple $(L,M,\xi)$; 
where $M$ is a rational homology 3-sphere
and $L$ is a null-homologous Legendrian $n$-component link.
We have the following surjective $\F$-linear map: \newpage
\[F:cCFL^-(D,\mathfrak t_{\xi})\xlongrightarrow{U=1}\widehat{CF}(D,\mathfrak t_{\xi})\:;\]
which is given by setting $U$ equals to 1.

The map $F$ clearly commutes with the differentials and it is such that $F(\mathfrak L(D))=c(D)$. 
It is well-defined and surjective, because every
intersection point $x$ is such that $F(x)=x$ and
if $s_w(x)=\mathfrak t_{\xi}$ then $s_z(x)=\mathfrak t_{\xi}$, since $L$ is null-homologous. 
Moreover, it respects the gradings in the following sense.
\begin{lemma}
 \label{lemma:1}
 The map $F$ sends an element with bigrading $(d,s)$ into an element with
 Maslov grading $d-2s$.
\end{lemma}
\begin{proof}
 We have that $$M_z(F(x))=d_3\left(\overline M,\mathfrak\pi_z(x)\right)=
 d_3\left(\overline M,\mathfrak\pi_w(x)\right)-\mathfrak s_{w,z}(x)[S]=M(x)-2A(x)\:,$$
 where $S$ is a Seifert surface for $L$.
\end{proof}
The map $F$ induces $F_*$ in homology:
$$F_*:cHFL^-\left(\overline M,L,\mathfrak t_{\xi}\right)
\xlongrightarrow{U=1}\widehat{HF}\left(\overline M,\mathfrak t_{\xi}\right)
\otimes\left(\F_{(-1)}\oplus\F_{(0)}\right)^{\bigotimes(n-1)}\:.$$
Then we can prove the following two lemmas.
\begin{lemma}
 \label{lemma:2}
 The map $F$ defined above is such that $F_*(\mathfrak L(L,M,\xi))=\widehat c(M,\xi)\otimes(\textbf e_{-1})^{\bigotimes(n-1)}$,
 where $\textbf e_{-1}$ is the generator of $\F_{(-1)}$.
\end{lemma}
\begin{proof}
 Denote with $D_1$ a Legendrian Heegaard diagram for the standard Legendrian unknot in $(M,\xi)$. If we perform $n$
 consecutive stabilizations on 
 $D_1$ then we easily obtain a diagram $D_n$ for the standard Legendrian $n$-component unlink. 
 
 Since $D$ and $D_n$ are Legendrian Heegaard diagrams for the same contact manifold $(M,\xi)$ and they have the same 
 number of basepoints, from \cite{Ozsvath1} we know that
 \begin{equation}
  \label{multi}
  \widehat{HF}(D,\mathfrak t_{\xi})\cong\widehat{HF}\left(\overline M,\mathfrak t_{\xi}\right)\otimes\left(\F_{(-1)}
  \oplus\F_{(0)}\right)^{\bigotimes(n-1)}\:.
 \end{equation}
 Moreover, a little variation of the maps $\Psi$ that we define in Section \ref{section:four} tells us that the
 isomorphism in Equation \eqref{multi} also identifies the class $[c(D)]$ with  
 $\widehat c(M,\xi)\otimes(\textbf e_{-1})^{\bigotimes(n-1)}$.
\end{proof}
Before proving the second lemma, we recall that, 
since the link Floer homology group is an 
$\F[U]$-module and $\F[U]$ is a principal ideal domain, we have 
$cHFL^-\left(\overline M,L,\mathfrak t_{\xi}\right)\cong\F[U]^r\oplus T$; where $r$ is
an integer and $T$ is the torsion $\F[U]$-module.
\begin{lemma}
 \label{lemma:map}
 The following two statements hold:
  \begin{enumerate}
   \item $F(x)=0$ and $x$ is homogeneous with respect to the Alexander grading if and only if $x=0$;
   \item $F_*[x]=[0]$ and $[x_i]$ is homogeneous with respect to the Alexander grading, for every $i=1,...,r$, if
         and only if $[x]$ is torsion.
         Here the $[x_i]$'s are a decomposition of $[x]$ in the torsion-free quotient of 
         $cHFL^-\left(\overline M,L,\mathfrak t_{\xi}\right)$.
  \end{enumerate}
\end{lemma}
\begin{proof}
 \begin{enumerate}
  \item The if implication is trivial. For the only if, suppose that $F(x)=0$; this gives that \newpage
        \[x=(1+U)\lambda_1(U)y_1+...+(1+U)\lambda_t(U)y_t\:,\]
        where $\lambda_i(U)\in\F[U]$ for every $i=1,...,t$  and $y_1,...,y_t$ are all the intersection points that induce the 
        $\text{Spin}^c$ structure 
        $\mathfrak t_{\xi}$.
        
        Since each $y_i$ is homogeneous and the monomial $U$ drops the Alexander grading by one, we have that $\lambda_i(U)=0$ 
        for every $i=1,...,t$ and then $x=0$.  
  \item Again the if implication is trivial. Now we have that $$[x]=\sum_{i=1}^r[x_i]+[x]_T\:,$$ where $[x]_T$ is the 
        projection of $[x]$ on the torsion submodule $T$.
        Since $F_*[x]=[0]$, we have \[[x]=(1+U)[z]=(1+U)\lambda'_1(U)[z_1]+...+(1+U)\lambda'_r(U)[z_r]+[x]_T\:;\] where one has
        $[x_i]=(1+U)\lambda'_i(U)[z_i]$ for 
        every $i=1,...,r$ and the $[z_i]'s$ are a homogeneous basis of the torsion-free quotient.
        
        The same argument of 1 implies that the polynomials $\lambda'_i(U)$ are all zero and then $[x]=[x]_T$.
 \end{enumerate}
\end{proof}
Now we use Lemma \ref{lemma:map} to show that there is a correspondence between the torsion-free quotient of the link Floer 
homology group and the multi-pointed hat Heegaard Floer homology.
\begin{teo}
 \label{teo:iso}
 If $L$ is a Legendrian $n$-component link in $(M,\xi)$ then there exists an isomorphism of $\F[U]$-modules 
 $$\begin{aligned}
    \dfrac{cHFL^-\left(\overline M,L,\mathfrak t_{\xi}\right)}{T}\longrightarrow &\\
    \left(\widehat{HF}\left(\overline M,\mathfrak t_{\xi}\right)\otimes_{\F}\F[U]\right)&
    \otimes_{\F[U]}\left(\F[U]_{(-1)}\oplus\F[U]_{(0)}\right)^{\bigotimes(n-1)}\:;
   \end{aligned}$$ which sends a homology class of bigrading $(d,s)$ into one of Maslov grading $d-2s$.   
\end{teo}
\begin{proof}
 We just have to show that $F_*$ sends $\{[z_1],...,[z_r]\}$, a homogeneous $\F[U]$-basis of the 
 torsion-free quotient of
 $cHFL^-\left(\overline M,L,\mathfrak t_{\xi}\right)$, into an $\F$-basis of 
 $\mathcal X=\widehat{HF}\left(\overline M,\mathfrak t_{\xi}\right)
 \otimes\left(\F_{(-1)}\oplus\F_{(0)}\right)^{\bigotimes(n-1)}$.
 Statement 1 in Lemma \ref{lemma:map} tells us that $F_*$ is surjective. In fact, if $[y]\in\mathcal X$ then one has 
 $0=\widehat\partial_zy=F(\partial^-x)$; where $F(x)=y$.
 We apply Lemma \ref{lemma:map} to $\partial^-(x)$, since we can suppose that both $x$ and $y$ are homogeneous, and we obtain 
 that $\partial^-(x)=0$; then $[x]$ is indeed a homology class. At this point, it is easy to see that the set $\left\{F_*[z_1],...,F_*[z_r]\right\}$ is a system of 
 generators of $\mathcal X$.
 
 In order to prove that $F_*[z_1],...,F_*[z_r]$ are also linearly independent in $\mathcal X$
 we suppose that there is a 
 subset $\left\{F_*[z_{i_1}],...,F_*[z_{i_k}]\right\}$ such that 
 $$F_*[z_{i_1}]+...+F_*[z_{i_k}]=F_*[z_{i_1}+...+z_{i_k}]=[0]\:.$$ Then we apply Statement 2 of lemma \ref{lemma:map} to 
 $[z_{i_1}+...+z_{i_k}]$ and we obtain that it is a torsion class. This is a contradiction, because $[z_{i_1}],...,[z_{i_k}]$ 
 are part of an $\F[U]$-basis of a torsion-free $\F[U]$-module.
\end{proof}
Lemmas \ref{lemma:1} and \ref{lemma:2} and Theorem \ref{teo:iso} immediately give the following corollary.
\begin{cor}
 \label{cor:vanish}
 The Legendrian invariant $\mathfrak L(L,M,\xi)$ is non-torsion if and only if the contact invariant $\widehat c(M,\xi)$ is 
 non zero.
 Furthermore, if $D$ is a Legendrian Heegaard diagram compatible with $(L,M,\xi)$ then the Maslov and Alexander grading of the 
 element $\mathfrak L(D)$ are related
 by the following equality:
 $$M(\mathfrak L(D))=-d_3(M,\xi)+2A(\mathfrak L(D))+1-n\:,$$ where $n$ is the number of component of $L$.
\end{cor}
In particular, this corollary says that the Legendrian 
link invariant $\mathfrak L$ is always
a torsion class if $\xi$ is overtwisted and always non-torsion if $(M,\xi)$ is strongly
symplectically fillable. In fact, from 
\cite{Ozsvath3} we know that $\widehat c(M,\xi)$ is zero in the first case and non-zero in the second one.

\subsection{Relation with classical Legendrian invariants}
We proved that the isomorphism class $\mathfrak L(L,M,\xi)$ is a Legendrian 
invariant. This is true also for 
the Alexander (and Maslov) grading of the element $\mathfrak L(D)\in cCFL^-(D,\mathfrak t_{\xi})$, that we 
denote with $A(\mathfrak L(D))$. 
From Corollary \ref{cor:vanish} we know that $\widehat c(M,\xi)\neq[0]$  
implies that $\mathfrak L(L,M,\xi)$ is non-torsion, hence it determines $A(\mathfrak L(D))$. On the other hand, if 
$\widehat c(M,\xi)$ is zero then a priori
the gradings of the element $\mathfrak L(D)$ could give more information.
Starting from these observations, we want to express the value of $A(\mathfrak L(D))$ in terms of the other known invariants 
of the Legendrian link $L$. 
Note that Corollary \ref{cor:vanish} also tells us that the Maslov grading
of $\mathfrak L(D)$ is determined, once we know $A(\mathfrak L(D))$. 

First we recall that, from the definitions of Thurston-Bennequin and rotation number, we have 
\[\text{tb}(L)=\sum_{i=1}^n\text{tb}_i(L)\:,\:\:\:\:\:\text{ where }
\text{tb}_i(L)=\text{tb}(L_i)+\text{lk}(L_i,L\setminus L_i)\]
and \[\text{rot}(L)=\sum_{i=1}^n\text{rot}_i(L)\:,\:\:\:\:\:\text{ where }\text{rot}_i(L)=\text{rot}(L_i)\:.\]
\begin{teo}
 \label{teo:ball}
 Consider $L\hookrightarrow(M,\xi)$ a null-homologous Legendrian 
 $n$-component link in a rational homology contact $3$-sphere and $D$ a 
 Legendrian Heegaard diagram, that comes from 
 an open book compatible with $(L,M,\xi)$. 
 Then we have that 
 \[(\mathfrak L(D))=\dfrac{\emph{tb}(L)-\emph{rot}(L)+n}{2}\:\:\:\:\:\text{and}\:\:\:\:\:M(\mathfrak L(D))=-d_3(M,\xi)+
 \emph{tb}(L)-\emph{rot}(L)+1\:.\]
\end{teo}
\begin{proof}
 If $L$ is a knot then the claim has been proved by Ozsv\'ath and Stipsicz in \cite{O-S} (Theorem 1.6). 
 At this point, in order to obtain the claim for links, we need to relate $A(\mathfrak L(D)$ with the Alexander grading
 of the Legendrian invariants of the components $L_i$ of $L$. 
 
 A Legendrian Heegaard diagram $D_i$ for the knot $L_i$ is easily gotten from $D$ by removing some curves and basepoints.
 We denote the intersection point representing the Legendrian invariant of $L_i$ with $\mathfrak L(D_i)$.
 Then we have \newpage
 \[\begin{aligned}
   A(\mathfrak L(D))&=\sum_{i=1}^n\left(A\left(\mathfrak L(D_i)\right)+
   \dfrac{1}{2}\text{lk}(L_i,L\setminus L_i)\right)=\\
   &=\sum_{i=1}^n\left(\frac{\text{tb}(L_i)-\text{rot}(L_i)+1}{2}+\dfrac{\text{lk}(L_i,L\setminus L_i)}{2}\right)=
   \sum_{i=1}^n\dfrac{\text{tb}_i(L)-\text{rot}_i(L)+1}{2}=\\
   &=\dfrac{\text{tb}(L)-\text{rot}(L)+n}{2}\:.
  \end{aligned}\]
\end{proof}

\subsection{The link Floer homology group of a Legendrian connected sum}  
Take two null-homologous Legendrian oriented links 
$L_1,L_2$, respectively in the connected contact 
manifolds $(M_1,\xi_1)$ and $(M_2,\xi_2)$. We can define a Legendrian connected sum of the two links \cite{Etnyre3}
and we denote it with 
$L_1\#L_2\hookrightarrow(M_1\#M_2,\xi_1\#\xi_2)$. 

While the contact 3-manifold $(M_1\#M_2,\xi_1\#\xi_2)$ is uniquely defined, the Legendrian link $L_1\#L_2$ depends on the
choice of the components used to perform the connected sum. Moreover, we have the following properties \cite{Etnyre,Etnyre3}:
\begin{itemize}
 \item $d_3(M_1\#M_2,\xi_1\#\xi_2)=d_3(M_1,\xi_1)+d_3(M_2,\xi_2)$;
 \item $\text{tb}(L_1\#L_2)=\text{tb}(L_1)+\text{tb}(L_2)+1$;
 \item $\text{rot}(L_1\#L_2)=\text{rot}(L_1)+\text{rot}(L_2)$. 
\end{itemize}
Let us consider two adapted open book decompositions $(B_i,\pi_i,A_i)$, compatible with the triples $(L_i,M_i,\xi_i)$. We can
define a third open book $(B,\pi,A)$, for the manifold $M_1\#M_2$, with
the property that $\pi^{-1}(1)$ is a Murasugi sum of the pages $\pi_1^{-1}(1)$ and $\pi_2^{-1}(1)$; see \cite{Etnyre} for the 
definition. The resulting open book
is compatible with the triple $(L_1\#L_2,M_1\#M_2,\xi_1\#\xi_2)$; where the Murasugi sum is done along the components involved 
in the connected sum.

Denote with $D_1$,$D_2$ and $D$ the Legendrian Heegaard diagrams obtained from the open book decompositions that we introduced
before. Then, we have the following theorem.
\begin{teo}
 \label{teo:connected_sum}
 For every $\text{Spin}^c$ structure on $M_1$ and $M_2$ there is a chain map 
 $$cCFL^-(D,\mathfrak t_1\#\mathfrak t_2)\longrightarrow
 cCFL^-(D_1,\mathfrak t_{1})\otimes_{\F[U]} cCFL^-(D_2,\mathfrak t_{2})$$ that preserves the bigrading and the element 
 $\mathfrak L(D)$ is sent into $\mathfrak L(D_1)\otimes\mathfrak L(D_2)$.
 
 Furthermore, this map induces an isomorphism in homology; which means that
 $$\begin{aligned}
 cHFL^-\left(\overline M_1\#\overline M_2,L_1\#L_2,\mathfrak t_{\xi_1\#\xi_2}\right)&\cong\\
 cHFL^-\left(\overline M_1,L_1,\mathfrak t_{\xi_1}\right)&\otimes_{\F[U]} 
 cHFL^-\left(\overline M_2,L_2,\mathfrak t_{\xi_2}\right)\end{aligned}$$ as 
 bigraded $\F[U]$-modules and 
 $$\mathfrak L(L_1\#L_2,M_1\#M_2,\xi_1\#\xi_2)=\mathfrak L(L_1,M_1,\xi_1)\otimes\mathfrak L(L_2,M_2,\xi_2)\:.$$
\end{teo}
\begin{proof}
 It follows from \cite{LOSS} Section 7, \cite{Ozsvath2} Section 7 and \cite{Ozsvath} Section 11.
\end{proof}
We note that the link Floer homology group of the connected sum does not depend on the choice of the components.  
In particular, this means that we can compute
the link Floer homology group and the Legendrian invariant of a disjoint union $L_1\sqcup L_2$. See \cite{C-M}
for the definition.
\begin{prop}
 If we denote with $\bigcirc_2$ a smooth $2$-component unlink and with $\mathcal O_2$ the Legendrian $2$-component unlink
 in $(S^3,\xi_{\text{st}})$ such that $\emph{tb}(\mathcal O_2)=-2$ then we have
 $$\begin{aligned}
    cHFL^-\left(\overline M,L_1\sqcup L_2,\mathfrak t_{\xi}\right)&\cong\\ 
    cHFL^-\left(\overline M_1,L_1,\mathfrak t_{\xi_1}\right)
  \otimes_{\F[U]}cHFL^-&\left(\overline M_2,L_2,\mathfrak t_{\xi_2}\right)
  \otimes_{\F[U]}cHFL^-(\bigcirc_2)\end{aligned}$$
 and 
 $\mathfrak L(L_1\sqcup L_2,M,\xi)=
 \mathfrak L(L_1,M_1,\xi_1)\otimes\mathfrak L(L_2,M_2,\xi_2)\otimes\mathfrak L(\mathcal O_2)$. 
\end{prop}
\begin{proof}
 We just apply Theorem \ref{teo:connected_sum} twice, each time on one of the two components of $\mathcal O_2$. 
\end{proof}
The homology group $cHFL^-(\bigcirc_2)$ is isomorphic, as bigraded 
$\F[U]$-module, to $\F[U]_{(-1,0)}\oplus\F[U]_{(0,0)}$ (a proof can be found in \cite{Book}).
Furthermore, Theorem \ref{teo:ball} tells us that $$\mathfrak L(\mathcal O_2)=\textbf e_{-1,0}\:,$$ that is the generator of 
$\F[U]$ with bigrading $(-1,0)$. 

This means that, if $\widehat c(M,\xi)$ is non-zero, we have that
$$M(\mathfrak L(L_1\sqcup L_2,M,\xi))=M(\mathfrak L(L_1,M_1,\xi_1))+M(\mathfrak L(L_2,M_2,\xi_2))-1$$ and
$$A(\mathfrak L(L_1\sqcup L_2,M,\xi))=A(\mathfrak L(L_1,M_1,\xi_1))+A(\mathfrak L(L_2,M_2,\xi_2))\:.$$

We also observe that:
\begin{itemize}
 \item $\text{tb}(L_1\sqcup L_2)=\text{tb}(L_1)+\text{tb}(L_2)$;
 \item $\text{rot}(L_1\sqcup L_2)=\text{rot}(L_1)+\text{rot}(L_2)$.
\end{itemize}

\subsection{Stabilizations of a Legendrian link}
We know that Legendrian links in the tight $S^3$ can be represented with front projections, see \cite{Etnyre4} for more 
details.
Then we define positive (negative) stabilization of a Legendrian link $L$ in $(S^3,\xi_{\text{st}})$, with front projection 
$P$, the Legendrian link $L^{\pm}$ represented
by the front projection $P^{\pm}$; which is obtained by adding two consecutive downward (upward) cusps to $P$. 
Stabilizations are well-defined up to Legendrian isotopy, in the sense that they
do not depend on the choice of the point of $P$ where we add the new cusps.

At this point it is easy for us the define stabilizations in every contact manifold. 
\begin{figure}[ht]
  \centering
  \def\svgwidth{10cm}
  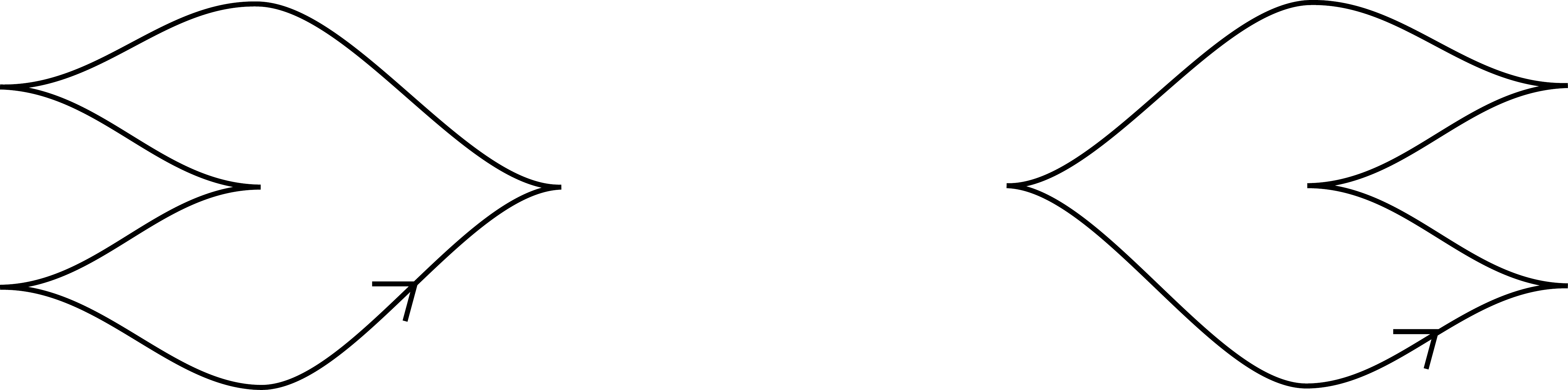   
  \caption{Front projection of $\mathcal O^+$ (left) and $\mathcal O^-$ (right).} 
  \label{Stabilizations}
\end{figure}
In fact we say that $L^\pm$ is the 
positive (negative) stabilization of $L$, 
a Legendrian link in $(M,\xi)$, if $L^\pm=L\#\mathcal O^\pm$; where $\mathcal O^\pm$ is the positive (negative) stabilization 
of the standard Legendrian unknot
$\mathcal O$ in $(S^3,\xi_{\text{st}})$. 
The Legendrian knots $\mathcal O^\pm$ are shown as front projections in Figure \ref{Stabilizations}.

The link type remains unchanged after stabilizations. The behaviour of the other classical invariants is given by
the following proposition.
\begin{prop}
 \label{prop:stabilizations}
 For every null-homologous Legendrian link $L$ in a contact $3$-manifold $(M,\xi)$ one has
 $$\emph{tb}\left(L^\pm\right)=\emph{tb}(L)-1\:\:\:\:\:\text{and}\:\:\:\:\:\emph{rot}\left(L^\pm\right)=\emph{rot}(L)\pm 1\:.$$
 Furthermore, we have that 
 $$\mathfrak L(L^+,M,\xi)=U\cdot\mathfrak L(L,M,\xi)\:\:\:\:\:\text{and}\:\:\:\:\:\mathfrak L(L^-,M,\xi)=\mathfrak L(L,M,\xi)$$
 in $cHFL^-\left(\overline M,L,\mathfrak t_{\xi}\right)$.
\end{prop}
\begin{proof}
 The first claim is a standard fact (see \cite{Etnyre4} for example). The second one follows from 
 Theorem \ref{teo:connected_sum}, which says that we just need
 to determine $\mathfrak L(\mathcal O^\pm)$, the fact that $cHFK^-(\bigcirc)\cong\F[U]_{(0,0)}$, 
 which says that $\mathfrak L(\mathcal O^\pm)$ 
 is fixed by the classical invariants
 of $\mathcal O^\pm$, and the first part of this proposition, which tells us that 
 $\text{tb}(\mathcal O^\pm)=-2$ and $\text{rot}(\mathcal O^\pm)=\pm 1$.
\end{proof}

\subsection{Loose Legendrian links}
Since $\widehat c(M,\xi)$ is always zero for overtwisted contact manifold, we have that the Legendrian link invariant 
$\mathfrak L$ is always torsion in this case. But Proposition \ref{prop:loose} says more in the case of loose Legendrian links.
\begin{proof}[Proof of Proposition \ref{prop:loose}]
 The complement of $L$ in $M$ contains an overtwisted disk $E$. Since $E$ is contractible, we can find a ball $U$ such that 
 $E\subset U\subset M\setminus L$. The restriction
 of $\xi$ to $U$ is obviously overtwisted; moreover, we can choose $E$ such that $\partial U$ has trivial dividing set. 
 Thus we have that $(M,\xi)=(M,\xi_1)\#(S^3,\xi')$; where $\xi_1$ concides with $\xi$
 near $L$ and $\xi'$ is an overtwisted structure on $S^3$.
 
 We now use the fact that the standard Legendrian unknot $\mathcal O$ is well-defined, up to Legendrian isotopy, and 
 Theorem \ref{teo:connected_sum} to say that
 $$\mathfrak L(L,M,\xi)=\mathfrak L(L,M,\xi_1)\otimes\mathfrak L(\mathcal O,S^3,\xi')\:.$$ But since
 $cHFK^-(\bigcirc)\cong\F[U]_{(0,0)}$, and then there is no torsion, we know that the Legendrian invariant
 of an unknot is zero in overtwisted 3-spheres. Then $\mathfrak L(L,M,\xi)$ is also zero.
\end{proof}
This proposition says something about stabilizations. In fact, in principle a stabilization of a non-loose 
Legendrian link $L\hookrightarrow(M,\xi)$
can be loose, but if $\mathfrak L(L,M,\xi)$ is non zero then all its negative stabilizations are also non-loose.

\section{The transverse case}
\label{section:six}
We recall that there is a way to associate a Legendrian oriented link to a transverse link and viceversa.
Given a Legendrian link $L$, we denote with $T_L$ the 
the transverse push-off of $L$, which is transverse. The construction is found in Section 2.9 in \cite{Etnyre4}; here we recall that transverse links have a canonical orientation induced by the contact form.
Any two transverse push-offs are transversely isotopic and then $T_L$ is uniquely defined, up to
transverse isotopy.

In the same way, if $T$ is a transverse link then we can define a Legendrian approximation $L_T$ of $T$. The procedure
is also described in Section 2.9 of \cite{Etnyre4}. The Legendrian link $L_T$ is not well-defined up to Legendrian
isotopy, but only up to negative stabilizations. Then from \cite{Epstein} we have the following theorem.
\begin{teo}[Epstein]
 \label{teo:epstein}
 Two transverse links in a contact manifold are transversely isotopic if and only if they admit Legendrian approximations
 which differ by negative stabilizations.
\end{teo}
The only classical invariant of a null-homologous transverse link $T$,
other than the smooth link type is the \emph{self-linking number} $\text{sl}(T)$.
In every $(M,\xi)$ rational homology contact 3-sphere, we define the number $\text{sl}(T)$ as follows:
$$\text{sl}(T)=\text{tb}(L_T)-\text{rot}(L_T)\:.$$
Clearly, from Theorem \ref{teo:epstein} and since negative stabilizations drop both the Thurston-Bennequin and rotation
number by one, we have that the self-linking number is a well-defined transverse invariant. Moreover, 
from the definition of self-linking number, we have the following
properties:
\begin{itemize}
 \item $\text{sl}(T_1\#T_2)=\text{sl}(T_1)+\text{sl}(T_2)+1$;
 \item $\text{sl}(T_1\sqcup T_2)=\text{sl}(T_1)+\text{sl}(T_2)$.
\end{itemize}
We can now define a transverse
invariant from link Floer homology by taking
$$\mathfrak T(T,M,\xi)=\mathfrak L(L_T,M,\xi)\:,$$
where $L_T$ is a Legendrian approximation of $T$.
The invariant $\mathfrak T$ has the same basic properties of $\mathfrak L$ that we recall in the following 
theorem.
\begin{teo}
 \label{teo:transverse}
 The isomorphism class $\mathfrak T(T,M,\xi)$ in $cHFL^-\left(\overline M,T,\mathfrak t_{\xi}\right)$ is a transverse
 link invariant. 
 If $n$ is the number of components of $T$ we have that the bigrading of $\mathfrak T$ is 
 \[(\mathfrak T(T,M,\xi))=\dfrac{\emph{sl}(T)+n}{2}\:\:\:\:\:\text{and}\:\:\:\:\:M(\mathfrak T(T,M,\xi))=-d_3(M,\xi)+
 \emph{sl}(T)+1\:.\] 
 Furthermore, one has
 \[\mathfrak T(T_1\#T_2,M_1\# M_2,\xi_1\#\xi_2)=\mathfrak T(T_1,M_1,\xi_1)\otimes\mathfrak T(T_2,M_2,\xi_2)\:.\]
\end{teo}
\begin{proof}
 Proposition \ref{prop:stabilizations} and Theorem \ref{teo:epstein} tell us that $\mathfrak T(T,M,\xi)$ is an
 invariant. The other properties follow from Theorem \ref{teo:ball}, the definition of self-linking number and the fact
 that the operations of Legendrian approximation and connected sum commute. 
\end{proof}
In the case of knots the invariant $\mathfrak T$ has been introduced first in \cite{LOSS}.

\section{Applications}
\label{section:seven}
\subsection{A different version of the Legendrian invariant}
Let us consider a Legendrian Heegaard diagram $D$, given by an open book compatible with a triple $(L,M,\xi)$, where $M$ is a 
rational homology 3-sphere and $L$
is a null-homologous Legendrian $n$-component oriented link with link type $\mathcal L$. We recall that, when $M$ admits a diffeomorphism that reverses the orientation, we can identify
$(\overline M,\mathcal L)$ with $(M,\mathcal L^*)$. We denote by $\mathcal L^*$ the \emph{mirror image} of the oriented link type $\mathcal L$.

We can define another chain complex by taking the $\F$-vector space
\[\widehat{CFL}(D,\mathfrak t)=\dfrac{cCFL^-(D,\mathfrak t)}{U=0}\] for every $\text{Spin}^c$ structure $\mathfrak t$ on $M$.
The corresponding differential is $\widehat\partial=\partial^-\lvert_{U=0}$.
We obtain the hat link Floer homology group 
\[\widehat{HFL}(D,\mathfrak t)=\bigoplus_{d,s\in\Q}\widehat{HFL}_{d,s}(D,\mathfrak t)\:;\]
given by
\[\widehat{HFL}_{d,s}(D,\mathfrak t)=\dfrac{\text{Ker }\widehat\partial_{d,s}}{\text{Im }\widehat\partial_{d+1,s}}\:.\]
The group $\widehat{HFL}(D,\mathfrak t)$ is a finite dimensional, bigraded $\F$-vector space and its
isomorphism type is invariant under smooth isotopy of the link $L$ \cite{Ozsvath2}.
Hence, we can denote $\widehat{HFL}(D,\mathfrak t)$ with $\widehat{HFL}\left(\overline M,L,\mathfrak t\right)$.

The intersection point $\mathfrak L(D)$ is a cycle also in $\widehat{CFL}(D,\mathfrak t)$ and it determines the $\text{Spin}^c$
structure $\mathfrak t_{\xi}$. Then we have the following theorem.
\begin{teo}
 The equivalence class of $(\widehat{HFL}\left(\overline M,L,\mathfrak t_{\xi}\right),[\mathfrak L(D)])$
 is a Legendrian invariant of $(L,M,\xi)$ and we denote it 
 with $\widehat{\mathfrak L}(L,M,\xi)$.
 Furthermore, if $\widehat{\mathfrak L}(L,M,\xi)$ is non-zero then $\mathfrak L(L,M,\xi)$ is also non-zero.
 
 For a null-homologous transverse link $T\hookrightarrow(M,\xi)$ 
 we have that $\widehat{\mathfrak T}(T,M,\xi)=\widehat{\mathfrak L}(T_L,M,\xi)$, where
 $T_L$ is a Legendrian push-off of $T$, is a transverse invariant of $T$ and it has the same non-vanishing property of 
 $\widehat{\mathfrak L}$.
\end{teo}
The proof of this theorem is the same as the one of
Theorems \ref{teo:main} and \ref{teo:transverse}, except for the non-vanishing 
property, which follows from the fact that $\widehat{CFL}(D,\mathfrak t)$ is a quotient of $cCFL^-(D,\mathfrak t)$.

The invariant $\widehat{\mathfrak L}(L,M,\xi)$ can be refined using a naturality property of the link Floer homology
group of a connected sum. Suppose that $L$ is a 
Legendrian oriented link, with link type $\mathcal L$, in a contact 3-sphere $(S^3,\xi)$ such that $\widehat{\mathfrak L}(L,S^3,\xi)\neq[0]$. 
Let $S$ be a convex, splitting sphere with connected dividing set, which intersects $L$ transversely in exactly two points. Such a splitting sphere
expresses $L$ as a connected sum of two links $L_1$ and $L_2$.

Since $L=L_1\#_S\:L_2$ then its hat Heegaard Floer homoloy group admits the splitting
\[\widehat{HFL}(\mathcal L^*)\cong\widehat{HFL}\left(\mathcal L_1^*\right)\otimes_{\F}\widehat{HFL}
\left(\mathcal L_2^*\right)\:,\] where the mirror images appear because $S^3$ has a diffeomorphism that reverses the orientation.

The Alexander grading of $\widehat{\mathfrak L}(L,S^3,\xi)$ is well-defined, because we suppose that the invariant is non-zero.
Moreover, we have that
$$A\left(\widehat{\mathfrak L}(L,S^3,\xi)\right)=
A\left(\widehat{\mathfrak L}(L_1,S^3,\xi_1)\right)+A\left(\widehat{\mathfrak L}(L_2,S^3,\xi_2)\right)\:.$$
The pair $(s_1,s_2)$, where $$s_i=A\left(\widehat{\mathfrak L}(L_i,S^3,\xi_i)\right)$$ for $i=1,2$, is called \emph{Alexander pair} \index{Alexander pair} of
$\widehat{\mathfrak L}(L,S^3,\xi)$ respect to $S$ and we denote it with $A_S(L,S^3,\xi)$. We have that the Alexander pair is an invariant of
$L$ in the sense of the following theorem.
\begin{teo}
 \label{teo:refinement}
 Suppose that $L$ is a Legendrian link in $(S^3,\xi)$ such that $\widehat{\mathfrak L}(L,S^3,\xi)$ is non-zero. 
 We also assume that there are two convex, splitting spheres $S_1$ and $S_2$, which decompose $L$ as Legendrian connected sums, such that
 we can find a smooth isotopy of $M$ that fix $L$ and sends $S_1$ into $S_2$. 
 
 Then the two Alexander pairs of $\widehat{\mathfrak L}(L,S^3,\xi)$, respect to $S_1$ and $S_2$, coincide, which means that
 $A_{S_1}(L,S^3,\xi)=A_{S_2}(L,S^3,\xi)$.
\end{teo}
\begin{proof}
 The proof is a link version of the one of Theorems 8.4 and 9.1 in \cite{LOSS}. 
\end{proof}
There is a version of Theorem \ref{teo:refinement} for transverse links. 
\begin{cor}
 \label{cor:refinement}
 Suppose that $T$ is a transverse link in $(S^3,\xi)$. Assume also that one of its Legendrian approximations $L_T$ respects the hypothesis of 
 Theorem
 \ref{teo:refinement}. Then one has \[A_{S_1}(T,S^3,\xi)=A_{S_2}(T,S^3,\xi)\:,\] where the Alexander pair is now defined as 
 $A_S(T,S^3,\xi)=A_S(L_T,S^3,\xi)$.
\end{cor}
\begin{proof}
 It is a consequence of the fact that Legendrian approximations of the same transverse link differ by negative stabilizations. Therefore,
 the Alexander gradings are the same because negative stabilizations do not change the invariant.
\end{proof}
The Alexander pair can be useful in distinguishing Legendrian and transverse links that are not isotopic.
\begin{prop}
 Suppose that $L_1$ and $L_2$ are smoothly isotopic Legendrian (transverse) links in $(S^3,\xi)$ which appear as follows. 
 Say $L_1\approx L_2$ is a $2$-component Legendrian (transverse) link, obtained from three Legendrian (transverse) 
        knots $K$, $H$ and $J$ with prime knot types, defined as 
        follows: take the connected
        sum of $K\#_S\:H$ with a (standard) positive Hopf link $\mathcal H_+$ and $J$, in the way that $K\#_S\:H$ is summed on the first component 
        of $\mathcal H_+$ and $J$ on the second one.        
 
 We have that if $A_S(L_1,S^3,\xi)\neq A_S(L_2,S^3,\xi)$ then $L_1$ is not Legendrian (transverse) isotopic to $L_2$.
\end{prop}
\begin{proof}
 It follows from Theorem \ref{teo:refinement} and Corollary \ref{cor:refinement} and the fact that, if there is a Legendrian (transverse)
 isotopy $F$ between $L_1$ and $L_2$, the isotopy $F$ is such that $F(S)=S'$ and we can smoothly isotope $S'$ onto $S$.
\end{proof}

\subsection{Non-loose Legendrian links with loose sub-links}
It is easy to prove that we can always find loose Legendrian links in every overtwisted contact 3-manifold; in fact Legendrian links inside
a Darboux ball need to be loose. \newpage
On the other hand, it is not harder to show that the same holds for non-loose Legendrian links. In fact, 
it is a known \cite{Etnyre} 
that every overtwisted contact 3-manifold is obtained from some $-1$-surgeries 
and exactly one $+1$-surgery on Legendrian knots in $(S^3,\xi_{\text{st}})$; then we just 
take the Legendrian link given by $n$ parallel contact push-offs of $J^+$, the knot where we perform the $+1$-surgery.
Moreover, it is easy to check that this link also has non-loose components.
 
A more interesting result is
to show that, under some hypothesis on $(M,\xi)$, we can also find non-split Legendrian $n$-components
links such that $\mathfrak L(L,M,\xi)\neq[0]$, which means that they are non-loose from Proposition \ref{prop:loose}, and
all of their sub-links are now loose.
We start by constructing Legendrian knots with non trivial invariant in all the overtwisted structures on $S^3$. 

Consider the family of Legendrian knots $L(j)$, where $j\geq 1$, given by the surgery diagram in Figure \ref{Nodo1}.
\begin{figure}[ht]
  \centering
  \def\svgwidth{12cm}
  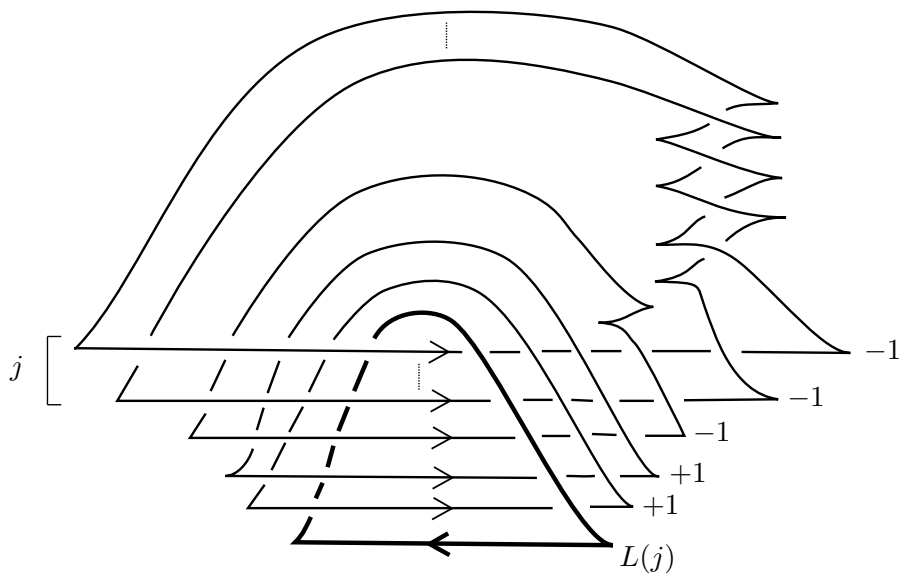   
  \caption{Contact surgery presentation for the Legendrian knot $L(j)$.}
  \label{Nodo1}
\end{figure}
Using Kirby calculus we easily see that $L(j)$ is a positive torus knot $T_{2,2j+1}$ in $S^3$. On the other hand, 
the Legendrian invariants of $L(j)$ and the contact structure where it lives are determined in \cite{LOSS} Chapter 6.
Namely, the knots in Figure \ref{Nodo1} are Legendrian knots 
in $\left(S^3,\xi_{1-2j}\right)$ and their invariants are:
\begin{itemize}
 \item $\text{tb}(L(j))=6+4(j-1)$;
 \item $\text{rot}(L(j))=7+6(j-1)$.
\end{itemize}
Furthermore, Proposition 6.2 in \cite{LOSS} tells us that $\widehat{\mathfrak L}(L(j),S^3,\xi_{1-2j})\neq[0]$ and then
$\mathfrak L(L(j),S^3,\xi_{1-2j})$ is a non-zero torsion class in $HFK^-(T_{2,-2j-1})$. 
Moreover, both have bigrading $(1,1-j)$.

Now we want to consider another family of Legendrian knots: the knots $L_{k,l}$, with $k,l\geq 0$, shown in Figure \ref{Nodo2}. 
\begin{figure}[ht]
  \centering
  \def\svgwidth{12cm}
  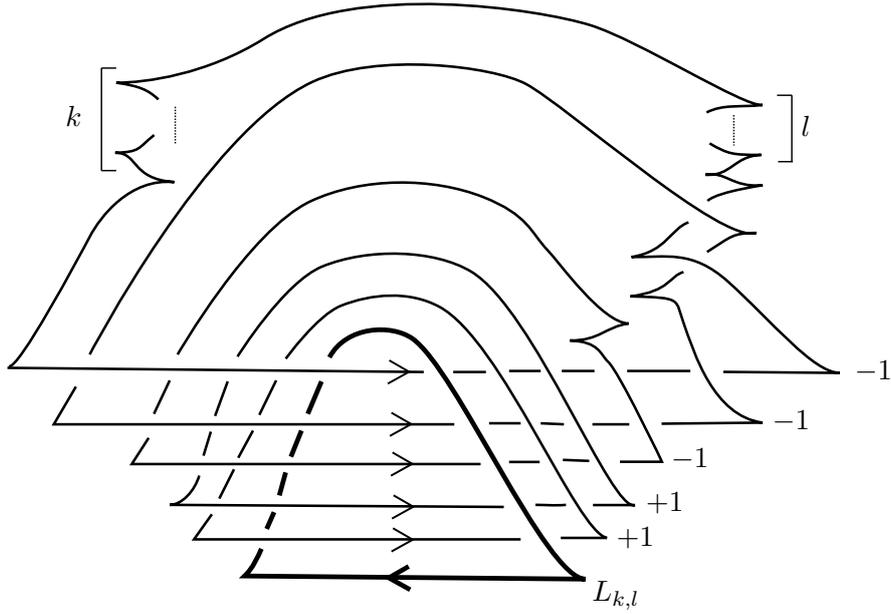   
  \caption{Contact surgery presentation for the Legendrian knot $L_{k,l}$.}
  \label{Nodo2}
\end{figure}
From \cite{LOSS} Chapter 6 we also know that $L_{k,l}$ is a negative torus knot $T_{2,-2k-2l-3}$ in 
$\left(S^3,\xi_{2l+2}\right)$ and 
its invariants are:
\begin{itemize}
 \item $\text{tb}(L_{k,l})=-6-4(k+l)$;
 \item $\text{rot}(L_{k,l})=-7-2k-6l$.
\end{itemize}
In this case, from Theorems 6.8, 6.9 and 6.10 in \cite{LOSS}, we have that the invariant 
$\widehat{\mathfrak L}$ of the Legendrian knots $L_{0,l}$, $L_{1,1}$ and $L_{1,2}$ is non-zero with bigrading $(-2k,1-k+l)$
in the homology group $\widehat{HFK}(T_{2,2k+2l+3})$.
Obviously, the fact that $\widehat{\mathfrak L}$ is non-zero
again implies that the same is true for the invariant $\mathfrak L$.

At this point, we define the Legendrian knots $K_i$, for every $i\in\Z$, in the following way:
$$K_i=\left\{\begin{aligned}
              &L(j)\#L(1)\:\:\:\:\:\:\:\text{if }i=-2j<0\\
 &L(j)\:\:\:\:\:\:\:\:\hspace{1cm}\text{if }i=1-2j<0\\
 &L_{0,j-1}\:\:\:\:\:\hspace{1cm}\text{if }i=2j>0\\
 &L_{0,j-1}\#L(1)\:\:\:\:\text{if }i=2j-1>0\\
 &L_{0,0}\#L(1)^2\:\:\:\:\:\:\text{if }i=0\:.
 \end{aligned}\right.$$
Then we have the following result; which is Theorem 7.2 in \cite{LOSS}.
\begin{prop}[Lisca, Ozsv\'ath, Stipsicz and Szab\'o]
 The Legendrian knot $K_i\hookrightarrow(S^3,\xi_i)$ is such that $\mathfrak L(K_i,S^3,\xi_i)\neq[0]$ and
 then it is non-loose for every $i\in\Z$.
\end{prop}
\begin{proof}
 It follows easily from the previous computation and the connected sum formula.
\end{proof}
We can now go back to links. Let us take an overtwisted 3-manifold $(M,\xi)$ such that there exists another contact 
structure $\zeta$ on $M$ with $\widehat c(M,\zeta)\neq[0]$ and $\mathfrak t_{\xi}=\mathfrak t_{\zeta}$; in particular $\zeta$
is tight. Consider $\mathcal O$ the standard Legendrian unknot in $(M,\zeta)$. We have that 
$\mathfrak L(\mathcal O,M,\zeta)$ coincides with $
\textbf e_{-d_3(M,\zeta),0}\neq[0]$ and 
the invariant is non-torsion; this is because 
$HFK^-\left(\overline M,\mathcal O,\mathfrak t_{\zeta}\right)\cong\F[U]_{(-d_3(M,\zeta),0)}$ and Theorem
\ref{teo:ball}.

Now we connect sum two copies of $L(1)$ to the standard positive Hopf link $H_+$ in $(S^3,\xist)$ in the way that
one copy is summed on the first component of $H_+$ and the other one on the second component, see Figure \ref{Hopf}.
\begin{figure}[ht]
  \centering
  \def\svgwidth{7.5cm}
  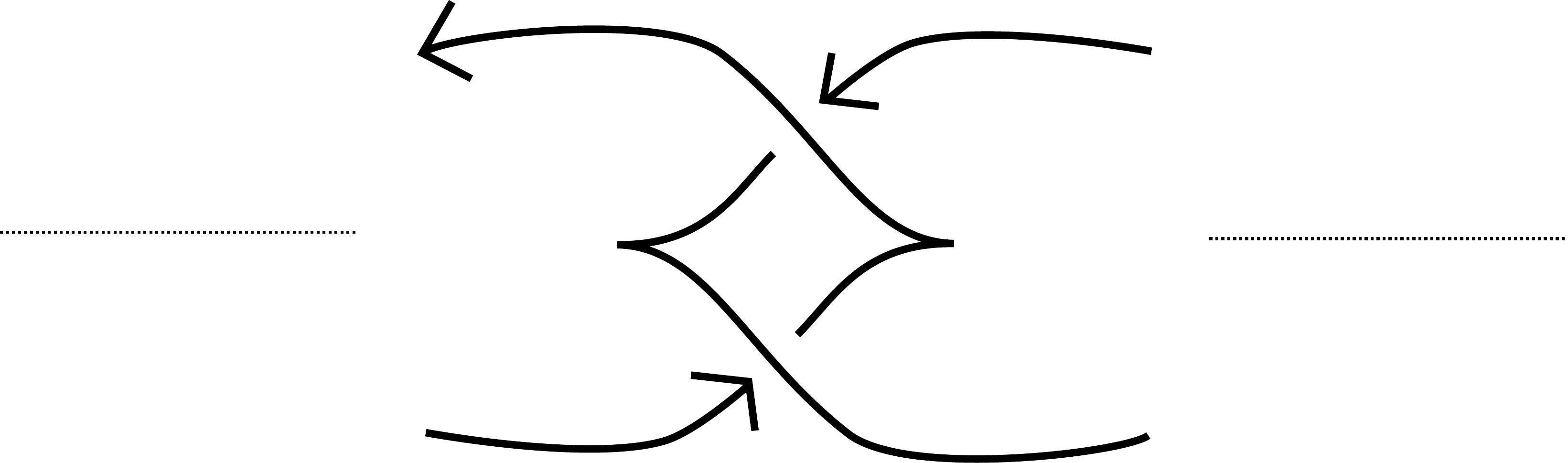   
  \caption{The two connected sums with the standard Legendrian positive Hopf link.}
  \label{Hopf}
\end{figure}
We repeat this procedure a total of 
$n-1$ times, where everytime the connected sum is performed in the way that the resulting link has components with form a single chain.

Let us denote these Legendrian links with $C_n$. Since $L(1)$ lives in $(S^3,\xi_{-1})$ we have that $C_n$ is an $n$-component Legendrian
link in $(S^3,\xi_{-n})$. 
The invariant $\mathfrak L(C_n,S^3,\xi_{-n})$ is
the tensor product of $n$ times $\mathfrak L(L(1),S^3,\xi_{-1})$ with $n-1$ times $\mathfrak L(H_+)$
the invariant in the standard $S^3$. An easy computation shows that $\mathfrak L(H_+)$ is the only non-torsion element
in the group $$cHFL^-(H_-)\cong\F[U]_{(0,0)}\oplus\F[U]_{(1,1)}\oplus\left(\dfrac{\F[U]}{U\cdot\F[U]}\right)_{(0,0)}$$
with bigrading $(1,1)$. This means that not only $\mathfrak L(H_+)$ is non-torsion, but it is also represented
by the top generator of one of the $\F[U]$-towers of $cHFL^-(H_-)$. 

Now we perform a connected sum between $C_n$ and the Legendrian unknot $\mathcal O$, introduced before.
Therefore, we can now see $C_n$ as a Legendrian link in $(M,\xi')$; where $\xi'$ is
an overtwisted structure such that $\mathfrak t_{\xi'}=\mathfrak t_{\zeta}=\mathfrak t_{\xi}$ and $d_3(M,\xi')=d_3(M,\zeta)-n$. 

If we perform another connected sum with the Legendrian knot $K_{d_3(M,\xi)-d_3(M,\zeta)+n}$ then we obtain the Legendrian link 
\[L=K_{d_3(M,\xi)-d_3(M,\zeta)+n}\#C_n\:;\] which is a link in $M$ equipped with a contact structure that has the same
Hopf invariant as $(M,\xi)$
and induces the same $\text{Spin}^c$ structure of $\xi$. By the Eliashberg's classification of overtwisted
structures \cite{Eliashberg2}, we conclude that $L$ is a Legendrian $n$-component link in $(M,\xi)$.
We can now prove Theorem
\ref{teo:nonloose}.
\begin{proof}[Proof of Theorem \ref{teo:nonloose}]
 We already saw that the link $L$ exists if the hypothesis of the theorem holds. So first we check that  
 $\mathfrak L(L,M,\xi)$ is non-zero. 
 In fact, the invariant is represented by the tensor product of a non-zero torsion element with $\mathfrak L(C_n,S^3,\xi_{-n})$.
 We are working with $\F[U]$-modules and we recall that 
 \[\F[U]\otimes_{\F[U]}\left(\dfrac{\F[U]}{U\cdot\F[U]}\right)\cong\dfrac{\F[U]}{U\cdot\F[U]}\] and, more precisely, in the 
 $\F[U]$-factor only the generator survives.
 
 Then, from what we said before, we have that $\mathfrak L(L,M,\xi)$ remains
 non-zero. In fact, the tensor products of the invariant $\mathfrak L$ of $K_n$ with $\mathfrak L(L(1),S^3,\xi_{-1})$ is non-zero because the hat versions 
 $\widehat{\mathcal L}$ are 
 non-zero in this case (\cite{LOSS}). Note that this conclusion would be false if instead we took the negative Hopf link. Now the invariant $\mathfrak L$
 does not lie in the top of an $\F[U]$-tower of the homology group and then it vanishes after the tensor product.
 
 This immediately implies that $\mathfrak T(T,M,\xi)$ is also non-zero and then the theorem holds for transverse links.
 Moreover, it is easy to see that the sub-links of $L$ are all loose; in fact, if  $L'$ is a sub-link of $L$ then there is at least one 
 component of $C_n$ which has been removed, say the $i$-th component. Since in an overtwisted 3-manifold we can always find an overtwisted disk disjoint 
 from a finite number of Darboux balls, this means that in the $i$-th $(S^3,\xi_{-1})$-summand we can find an 
 overtwisted disk that happens to lie in the complement of $L'$.
 
 It is only left to prove that $L$ is non-split. 
 From \cite{Colin} we know that the connected sum of two tight contact manifolds is still tight. 
 This implies that a non-loose Legendrian link is split if and only if its smooth link
 type is split. Hence, we just have to show that $L$ is non split as a smooth $n$-component link.
 But $L$ is a connected sum of torus links in a 3-ball inside $M$ and we know that $L$ is non-split as a link in $S^3$.
 Furthermore, if $L$ is split in $M$ then it would be split also in the 3-sphere and this is a contradiction.
\end{proof}

\subsection{Non-simple link types}
In the previous subsection we saw that, under some hypothesis, in an overtwisted 3-manifold $(M,\xi)$ we can find non-loose,
non-split Legendrian $n$-components links $L_n$. Consider the links $L_n'$ obtained as the connected sum 
of $L_n$ with the standard Legendrian unknot in $(S^3,\xi_0)$, where $\xi_0$ is the overtwisted $S^3$ with zero Hopf 
invariant. 

Since $(M,\xi)$ is already overtwisted, 
we have that $(M,\xi)\#(S^3,\xi_0)$ is contact isotopic to $(M,\xi)$.
This means that $L_n'$ is also a non-split 
Legendrian link in $(M,\xi)$, which is smoothly isotopic to $L_n$ for every $n\geq 1$, but
unlike $L_n$ it is clearly loose.
Each component of $L_n'$ has the same classical invariants of a component of $L_n$. Moreover, if $n\geq 2$ then
there is an overtwisted 
disk in their complement. From a result of Dymara in \cite{Dymara} the components of $L_n'$ are Legendrian
isotopic to the ones of $L_n$.
Hence, we have the follwing corollary.
\begin{cor}
 The link type of $L_n$ and $L_n'$ in $M$, which is denoted with $\mathcal L$, is both Legendrian and transverse non-simple.
\end{cor}
On the other hand, we can also find non-simple link types where the two
Legendrian and transverse representatives are non-loose.
\begin{prop}
 \label{prop:nonsimple}
 Let us consider
 the links $L_1=(L_{0,2}\#L_{1,2})\#H_+\#L(1)$ and $L_2=(L_{1,1}\#L_{0,3})\#H_+\#L(1)$
 in the contact manifold
 $(S^3,\xi_{11})$; where in $L_i$ the knots on the left
 are summed on the first component of $H_+$ and $L(1)$ on the second one.
 Then $L_1$ and $L_2$ are two non-loose, non-split Legendrian $2$-component links, with the same classical invariants and
 Legendrian isotopic components, but that are not Legendrian isotopic.
 
 In the same way, the transverse push offs of $L_1$ and $L_2$ are two non-loose, non-split transverse $2$-component links, 
 with the same classical invariants and
 transversely isotopic components, but that are not transversely isotopic.
\end{prop}
\begin{proof}
 We apply Theorem \ref{teo:refinement}. The Legendrian invariant of $L_1$ and $L_2$ is computed in \cite{LOSS} and it is
 non-zero; moreover, the Alexander pairs of $\widehat{\mathfrak L}(L_1,S^3,\xi_{11})$ and
 $\widehat{\mathfrak L}(L_2,S^3,\xi_{11})$ are different. The fact that the components are Legendrian isotopic follows
 from Dymara's result \cite{Dymara}. 
 The same argument proves the theorem in the transverse setting.
\end{proof}
Using the same construction, the refined version of $\widehat{\mathfrak L}$ and
$\widehat{\mathfrak T}$ can be applied to find such examples for 
links with more than two components in every contact manifold as in Theorem \ref{teo:nonloose}.
\begin{proof}[Proof of Theorem \ref{teo:last}]
 Let us take a standard Legendrian (transverse) 
 positive Hopf link $H_+$ in $(S^3,\xi_{\text{st}})$. On the first component of $H_+$ 
 we perform a connected sum with the knot $L_{0,2}\#L_{1,2}$ in one case and $L_{1,1}\#L_{0,3}$ in the other.
 While, on the second component of $H_+$, we sum a non-split Legendrian (transverse) $n$-component link in the 
 overtwisted manifold $(M,\xi')$, where $d_3(M,\xi')=d_3(M,\xi)-12$, with non-zero invariants;
 those links exist as we know from
 Theorem \ref{teo:nonloose}. We conclude by applying the same argument of the proof of Proposition \ref{prop:nonsimple}.
\end{proof}

\end{document}